\documentclass[12pt]{amsart}
\usepackage[numbers, square]{natbib}
\usepackage{amssymb}
\usepackage{amsmath}
\usepackage{amsfonts}
\usepackage{geometry}
\usepackage{bbm}
\usepackage{dsfont}
\usepackage{graphicx}
\usepackage{amsthm}
\usepackage{hyperref}
\usepackage[dvipsnames]{xcolor}
\usepackage{esint}

\providecommand{\U}[1]{\protect\rule{.1in}{.1in}}
\theoremstyle{plain}

\newtheorem{corollary}{Corollary}

\newtheorem{lemma}{Lemma}

\newtheorem{theorem}{Theorem}
\theoremstyle{remark}
\newtheorem{remark}{Remark}
\numberwithin{equation}{section}

\DeclareMathOperator{\supp}{supp}

\newcommand{\eps}{\varepsilon}

\newcommand{\R}{\ensuremath{\mathbb{R}}}

\begin{document}

\title[]
{Spectral inequalities for Schr\"odinger equations with  various potentials }
\author{ Jiuyi Zhu}
\address{
Department of Mathematics\\
Louisiana State University\\
Baton Rouge, LA 70803, USA\\
Email:  zhu@math.lsu.edu }

\date{}
%\dedicatory{}

\begin{abstract}
We study the spectral inequalities of  Schr\"odinger operators in the whole space for different potentials, which can be polynomial type growth or  vanishing at infinity. The spectral inequalities quantitatively depend on the density of the sensor sets with positive measure, growth rate of the potentials and  spectrum (or eigenvalues). One important component in the poof is the adaptation of propagation of smallness  argument for  gradients in \cite{LM18}. As an application, we apply the spectral inequalities to obtain quantitative observability inequalities for heat equations.
\end{abstract}

\maketitle
\section{Introduction}

The paper is devoted to the study of spectral inequalities for the Schr\"odinger operator $H=-\triangle +V(x)$ for different types of potentials $V(x)$ in $\mathbb R^n$. The spectral inequality is concerned with the control estimates for a linear combination of eigenfunctions. Let us start by discussing the spectral inequalities in a compact Riemannian manifold. The Laplace eigenfunction $\phi_k$ is given as
\begin{equation}
-\triangle_g \phi_{k} = \lambda_k\phi_{k} \quad \mbox{on} \ \mathcal{M},
\label{classs}
\end{equation}
where $(\mathcal{M},g)$ is a Riemannian manifold and $-\triangle_g$ is the Laplace-Beltrami operator on $\mathcal{M}$. We consider a linear combination of Laplace eigenfunctions $\phi_k$ with $\lambda_k\leq \lambda$ for some $\lambda>0$. That is, $\phi =\sum_{\lambda_k\leq \lambda} \alpha_k\phi_{k}$ for $\alpha_k\in \mathbb R$.
The following sharp version of spectral inequality 
\begin{align}
 \|\phi\|_{L^2(\mathcal{M})}\leq C_0 e^{C_1 \lambda^\frac{1}{2}}  \|\phi\|_{L^2(\Omega)}
 \label{spec-in-1}
\end{align}
was shown in \cite{LR}, \cite{JL}, where $\Omega$ is an open subset, and $C_0$ and $C_1$ depend on $\mathcal{M}$ and $\Omega$.
The spectral inequality  (\ref{spec-in-1})  was used to study the null-controllability problem for the corresponding heat equation in \cite{LR}, the Hausdorff measure of nodal sets for finite sums of eigenfunctions in \cite{JL}, and the null-controllability of thermoelasticity system in \cite{LZ}. See the book \cite{LLR22} for more extensive literature on the proof and applications of the spectral inequality (\ref{spec-in-1}).

In the context of control theory, it is also important to study 
the spectral inequalities for a nonnegative self-adjoint elliptic operator $H$ in $\mathbb R^n$. Then the spectral inequality takes the following form
\begin{align}
 \|\phi\|_{L^2(\mathbb R^n)}\leq C_0 e^{C_1 \lambda^\gamma}  \|\phi\|_{L^2(\Omega)} \quad \text{for any } \phi\in  Ran(P_\lambda(H)),
 \label{spec-in}
\end{align}
with some universal constants $\gamma\in (0,1), C_0>0, C_1>0$, where  $\Omega \subset \mathbb R^n$ is a measurable subset, $P_\lambda(H) = \chi_{(-\infty, \lambda)}(H)$ is the  spectral projection of $H$ in $\mathbb R^n$ and $Ran(P_\lambda(H))$ consists of finite sums of  eigenfunctions.

The spectral inequality (\ref{spec-in}) is in the same spirit of the following sharp doubling inequality of eigenfuntion $\phi_k$ in (\ref{classs})
\begin{equation}\|\phi_{k}\|_{L^2(\mathbb
B_{2r}(x))}\leq  e^{C{\sqrt{\lambda_k}}}\|\phi_{k}\|_{L^2(\mathbb
B_{r}(x))} \quad \text{for all } \mathbb B_{2r}(x)\subset \mathcal{M},
\label{doub}
\end{equation}
which quantitatively controls global information from local data. The inequality (\ref{doub}) was obtained in \cite{DF}  for Laplace eigenfunctions in (\ref{classs}), where $C$ depends only on $\mathcal{M}$. The doubling inequality (\ref{doub}) and its variants play an essential role in the study of the measure of nodal sets, see e.g. \cite{DF}, \cite{Lin}, \cite{L18} and the study of quantitative unique continuation properties, see e.g. \cite{Zh1}, \cite{K07}.

We study the spectral inequalities for the Schr\"odinger operator $H=-\triangle +V(x)$ in $ \R^n$ for potential $V(x)$ with different properties and $n\geq 1$. If
$\lim_{|x|\to \infty} V(x)=+\infty$, the inverse operator $H^{-1}$ is compact in $L^2(\R^n)$. Hence the spectrum of $H$ are discrete (called eigenvalues) with finite multiplicities and  $\lim_{k\to \infty} \lambda_k=\infty$. That is, 
$\phi_k$ is the eigenfunction of $H$ corresponding to the eigenvalue $\lambda_k$ satisifying
\begin{align}
-\triangle \phi_k +V(x)\phi_k=\lambda_k \phi_k \quad \text{ in } \mathbb{R}^n.
\label{eigen-k}
\end{align}
For $\phi\in Ran(P_\lambda(H))$, we can write
\begin{align}
\phi=\sum_{\lambda_k\leq \lambda} \alpha_k \phi_k, \quad \mbox{with}\ \ \alpha_k=\langle \phi_k, \phi\rangle.
\label{phi-s}
\end{align}
Note that $\{ \phi_k: \lambda_k\leq \lambda\}$ forms an orthogonal basis of $ Ran(P_\lambda(H))$. If the potential $V(x)$ grows as polynomials at infinity, the eigenfunctions are well localized and decaying exponentially. 

The spectral inequality (\ref{spec-in}) relies on the potential $V(x)$ and a certain given measurable set $\Omega$.
A measurable {sensor set} $\Omega \subset \R^n$ is called {efficient} if the spectral inequality \eqref{spec-in} holds with some $ \gamma\in (0,1)$. It was shown in e.g., \cite[Theorem 2.8] {NTTV1} that an efficient sensor set would guarantee the null-controllability for the corresponding heat equation in $\R^n$. 
Denote by $\Lambda_L(j)=j+(-\frac{L}{2}, \frac{L}{2})^n$ the cube with side length $L>0$ for $j=(j_1, j_2, \cdots, j_n)\in \mathbb Z^n$ and $\mathcal{B}_r(z)$ the ball centered at $z$ with radius $r$ in $\mathbb R^n$. $\mathcal{B}_r$ is denoted as the ball centered at origin with radius $r$. Let $\delta \in (0,1)$ and $\sigma\in [0,1)$. 
We introduce one type of the sensor sets $\Omega$ as
\begin{align}
   \mathcal{B}_{\delta^{1+|j|^{\sigma}}L}(z_j)\subset \Omega \cap \Lambda_L(j)
   \label{sensor-open}
\end{align}
for all $j\in \mathbb{Z}^n$ and  well distributed sequence $\{z_j\}$. A simple example of this type of sensor sets is
 \begin{align}
 \Omega=\bigcup_{j\in \mathbb Z^n} \mathcal{B}_{2^{-(1+|j|^\sigma)}}(j)
 \label{example}
 \end{align}
 for $\delta=\frac{1}{2}$ and $L=1$.
 If $H=-\triangle +V(x)$ with $V(x)\in C_{loc}^{0,1}(\mathbb R^n)$ satisfying the assumptions $$c_1(|x|-1)^{\beta_1}\leq V(x)+|DV|\leq c_2(|x|+1)^{\beta_2}$$ for some positive constants $c_1, c_2$ and $\beta_1\leq \beta_2$, we proved in \cite{ZZ23} the following spectral inequality 
\begin{align}
\|\phi\|_{L^2(\mathbb R^n)}\leq (\frac{1}{\delta})^{C\lambda^{\frac{\sigma}{\beta_1}+\frac{\beta_1}{2\beta_2}}} \|\phi\|_{L^2(\Omega)} \quad \text{for all } \phi \in Ran(P_\lambda(H))
\label{old-zz}
\end{align}
on the sensor sets $\Omega$ in (\ref{sensor-open}). We developed some new quantitative Carleman estimates that make use of the gradients of the potential $V(x)$. It turns out that sensor sets $\Omega$ in (\ref{sensor-open}) is efficient if 
$\frac{\sigma}{\beta_1}+\frac{\beta_1}{2\beta_2}<1$. In particular, if $V(x)=|x|^{\beta_1}$ for $\beta_1>0$, the sharp spectral inequality is shown
\begin{align}
\|\phi\|_{L^2(\mathbb R^n)}\leq (\frac{1}{\delta})^{C\lambda^{\frac{\sigma}{\beta_1}+\frac{1}{2}}} \|\phi\|_{L^2(\Omega)} \quad \text{for all } \phi \in Ran(P_\lambda(H)),
\label{sharp-spe}
\end{align}
which answered an open question in \cite{DSV} where a sub-optimal constant $(\frac{1}{\delta})^{C\lambda^{\frac{\sigma}{\beta_1}+\frac{\beta_1}{2\beta_2}}}$ in (\ref{old-zz}) was obtained. Note that $V(x)=|x|^{\beta_1}$ for $0<\beta_1<1$ is not Lipschitz continuous at origin, but its $L^\infty$ norm is bounded near the origin.

Instead of containing a ball in each $\Lambda_L(j)$, we study the sensor sets $\Omega$ with  positive measure 
\begin{align}
\frac{|\Omega\cap \Lambda_L(j) |}{|\Lambda_L(j)|}\geq \delta^{1+|j|^\sigma}
\label{geom-om}
\end{align}
for all $j\in \mathbb Z^n$ and $\sigma\in [0, 1)$ in this paper, where $|\cdot|$ denotes the Lebesgue measure. Without loss of generality, we may assume $L=1$. Note that the volume of sensor sets $\Omega$ in (\ref{geom-om}) is bounded, because of the presence of $\sigma>0$. If $\sigma =0$, the sensor sets are called thick sets, which has been used previously to study spectral inequalities for harmonic oscillator $H = -\triangle + |x|^2$ in e.g. \cite{MPS,BJP}.  The spectral inequality (\ref{spec-in}) was proved in \cite{DSV1} for harmonic oscillator  with $\gamma = \frac{\sigma}{2} + \frac12$.
These proofs rely on the real analyticity of $\phi_k$ and complex analysis estimates (i.e. Bernstein-type inequality). Later on, some similar  spectral inequalities \label{sharp-spe} were shown in \cite{AS} and \cite{M22}
for the Schr\"odinger operator $H=-\triangle +|x|^{\beta_1}$ with even integers $\beta_1\in 2\mathbb N$, which relies on the real analyticity of $\phi_k$ and analyticity arguments as well. 

In this paper, we aim to study the spectral inequality (\ref{spec-in}) under some general assumption of the potential $V(x)$ on the sensor sets (\ref{geom-om}) without using the analyticity of $\phi_k$.
Let us state our assumption on $V$. 

\textbf{Assumption (A)}: Assume the real-valued function  $V\in L^\infty_{\rm loc}(\R^n)$ and $V\in C^{0, 1}_{\rm loc}(\R^n\backslash \mathcal{B}_R)$ for some $R>0$. Furthermore,  $V$ satisfies the following two conditions:
\begin{itemize}
    \item There exist positive constants $c_1$, $C^0$, $\beta_1$ such that for all $x\in \R^n$,
    \begin{equation}\label{V.growth}
    c_1|x|^{\beta_1} -C^0\le V(x).
    \end{equation}
    \item  there exist positive constants $c_2$ and $\beta_2\geq \beta_1$ such that 
    \begin{equation}\label{V.bound}
    |V(x)| + |D V(x)| \le c_2(|x| + 1)^{\beta_2}.
\end{equation}
\end{itemize}

Under the assumption of (A), we are able to show the following theorem.
\begin{theorem}
Let $H=-\triangle +V(x)$. Assume  that $V$ satisfies Assumption (A) and $\Omega$ satisfies (\ref{geom-om}) with $L=1, \sigma\in [0, 1)$ and $\delta\in (0, \frac{1}{2})$. Then there exists a constant $C$ depending only on $\beta_1, \beta_2, C^0, c_1, c_2, R, \sigma$ and $n$ such that for $\lambda\geq 1$,
\begin{align}
\|\phi\|_{L^2(\mathbb R^n)}\leq (\frac{1}{\delta})^{C\lambda^{\frac{\sigma}{\beta_1}+\frac{\beta_2}{2\beta_1}}} \|\phi\|_{L^2(\Omega)} \quad \text{for all } \phi\in Ran(P_\lambda(H)).
\label{aim-res}
\end{align}
\label{th1}
\end{theorem}

%Throughout, we will use the notations $D_i = \frac{\partial }{\partial x_i}$ for partial derivatives and $D = (D_1,D_2,\cdots)$ for gradients.  
Obviously, our assumption (A) includes the particular case $V(x) = |x|^{\beta_1}$ for any $\beta_1>0$, and requires much less regularity assumptions of $V$ than the previous literature.
%This flexible assumption is sufficient for many applications, particularly 
%We should point out that the exponent $\frac43$ for $V_2$ in \eqref{V.bound} is consistent with literature; see \cite{NTTV,DSV}.
Our proof is different from the strategy in e.g. \cite{LR}, \cite{JL}, \cite{DSV}, \cite{ZZ23}, which apply certain local and global Carleman estimates, and is not the same as  e.g. \cite{DSV1}, \cite{AS} and \cite{M22}, which rely on analyticity of $\phi_k$ and Bernstein inequalities. We make use of the propagation of smallness arguments for the gradients on the sets of positive measure in \cite{LM18}. This idea has been used in \cite{BM21} for spectral inequality on the sets with positive measure on a compact manifold $\mathcal{M}$. See also \cite{BM21} for spectral inequalities in $\mathbb R^n$ without the potential $V(x)$. Our new difficulty is to take care of the presence of growing potential $V(x)$. Especially, $V(x)$ grows to infinity as $|x|\to \infty$. To overcome the difficulty, we apply the idea of exponential decay of eigenfunctions to show a doubling type inequality for $\phi$ in a large ball depending on $\lambda$. The assumption (A) includes the interesting case of polynomial growth potential $V(x)=|x|^{\beta_1}$ for $0<\beta_1<1$. However, it's Lipschitz norm blows up at the origin. We introduce a new way to tackle the singularity of Lipschitz norm in any local region and are still able to construct a second elliptic equation with Lispchitz leading coefficients. Then we incorporate the potential $V(x)$ into the leading coefficients of a second order elliptic equation without lower order terms. We obtain the propagation of smallness for the gradient for the elliptic equation in each cubes as the leading coefficients of the new second order elliptic equations will be Lipschtiz continuous and  well behaved in the cubes. We sum up the estimates on each cube to have the global estimates.

%An interesting case of polynomial growth potential is $V(x)=|x|^{\beta_1}$ for $0<\beta_1<1$. In this case, $V(x)=|x|^{\beta_1}$ satisfies the Assumption (A) for $|x|\geq 1$. However, it's Lipschitz norm blows up at the origin. We introduce a new way to tackle the singularity and are still able to construct a second elliptic equation with Lispchitz leading coefficients. Using the strategy in the proof of Theorem \ref{th1}, we can show that
%\begin{theorem}
%Let $H=-\triangle +|x|^{\beta_1}$ for $\beta_1>0$. Assume that $\Omega$ satisfies (\ref{geom-om}) with $L=1, \sigma\in [0, 1)$ and $\delta\in (0, \frac{1}{2})$. Then there exists a constant $C$ depending only on $\beta_1, \sigma$ and $n$ such that for $\lambda\geq 1$,
%\begin{align}
%\|\phi\|_{L^2(\mathbb R^n)}\leq (\frac{1}{\delta})^{C\lambda^{\frac{\sigma}{\beta_1}+\frac{1}{2}}} \|\phi\|_{L^2(\Omega)} \quad \text{for all } \phi\in Ran(P_\lambda(H)).
%\label{second-type}
%\end{align}
%\label{th1-1}
%\end{theorem}

%\begin{remark}
%Note that the strategy in Theorem \ref{th1-1} is used to tackle the failure of Lipscthiz norm  for $V(x)$ in any compact subset in $\mathbb R^n$. Combining the ideas in the proof of Theorem \ref{th1}) and Theorem  \ref{th1-1}, we can relax the Assumption (A) for $V(x)$. 
 %   Instead of having $V\in C^{0, 1}_{\rm loc}(\R^n)$, we can assume For such a $V(x)$, under the assumptions of (\ref{V.growth}), (\ref{V.bound})  and $\Omega$ in (\ref{geom-om}), the spectral inequality (\ref{aim-res}) holds.
%\end{remark}

In the second part of the paper, we study the spectral inequalities for 
\begin{align}
-\triangle \phi+V(x)\phi=\lambda \phi \quad \mbox{in} \quad \mathbb R^n
\end{align}
with a bounded potential $V(x)$. 
The assumption of the potential $V(x)$ is stated in the following.

\textbf{Assumption (B)}: Assume that the real-valued function  $V(x)$ is bounded and $\lim_{|x|\to \infty} V(x)=0$.
Thus, for some constant $C^0$, we have
\begin{align}
\|V\|_{L^\infty}\leq C^0.
\end{align}

Unlike the sensor sets assumption in (\ref{geom-om}), we study the measurable sensor sets $\Omega\subset \mathbb R^n$ satisfying the following property:
\begin{align}
\frac{|\Omega\cap \Lambda_L(j) |}{|\Lambda_L(j)|}\geq \delta
\label{geom-om-2}
\end{align}
for all $j\in \mathbb{Z}^n$. The sensor sets in (\ref{geom-om-2}) can be deduced from  (\ref{geom-om}) with $\sigma=0$. It is called the thick sets as $\Omega$ has infinite volume. 

The spectrum of $H=-\triangle +V(x)$ in Assumption (B) behaves in a different way from the polynomial growth potential $V(x)$.
The spectrum of $H=-\triangle +V(x)$ is of the form 
\begin{align}
    \sigma(H)=\sigma_{disc}(H) \cup \sigma_{ess}(H),
    \label{spec-c}
\end{align}
where $\sigma_{ess}(H)=[0, \infty)$ is the essential spectrum, $\sigma_{disc}(H)$ is the discrete eigenvalues which is of finite number and lies in $(-\infty, 0)$. Note the condition $\lim_{|x|\to \infty} V(x)=0$ in Assumption (B) is not used in the proof of Theorem \ref{th2} below. It is only used for the characterization of the spectrum in (\ref{spec-c}).
Since $V(x)$ is bounded with a lower bound, then there exists a positive constant $E_0$ such that $\inf\sigma(H)\geq -E_0$. See e.g. \cite{K18}. We are able to show the following spectral inequalities.
\begin{theorem}
Assume  that $V$ satisfies Assumption (B) and $\Omega$ satisfies (\ref{geom-om-2}) with $L=1$ and $\delta\in (0, \frac{1}{2})$. For any $f\in L^2(\mathbb R^n)$, there exists a constant $C$ depending only on $C^0$ and $n$ such that
\begin{align}
\|\mathbb{I}_{\mu}(f)\|_{L^2(\mathbb R^n)}\leq (\frac{1}{\delta})^{C(\sqrt{|\mu|}+1)} \|\mathbb{I}_{\mu}(f)\|_{L^2(\Omega)},
\label{spectral-conti}
\end{align}
where the spectral projection $\mathbb{I}_{\mu}(f)$ is given in (\ref{proj-mu}).
\label{th2}
\end{theorem}

By assuming $V(x)$ is analytic and has some decay estimates at infinity for the holomorphical extension of $V(x)$, a similar result as (\ref{spectral-conti}) was obtained in \cite{LeMo19}, which relies on complex analysis techniques. Compared with \cite{LeMo19}, we get rid of the analyticity assumptions for $V(x)$ and require much weaker assumptions for $V(x)$ in Theorem \ref{th2}. Our proof is based on 
the consequence of propagation of smallness in \cite{LM18} and the strategy in the proof of Theorem \ref{th1}.

The spectral inequality (\ref{spectral-conti}) is closely connected to the uncertainty principle, which says that a function can not be localized both in space and in the frequency variable. 
Under the condition (\ref{geom-om-2}), the Logvenenko-Sereda inequality states that, for any $g\in L^2(\mathbb R^n)$,
\begin{align}
\|g\|_{L^2(\mathbb R^n)}\leq (\frac{1}{\delta})^{C(\sqrt{\mu}+1)} \|g\|_{L^2(\Omega)}, \quad \mbox{if} \ \supp\hat{g}\subset \mathcal{B}_{\mu}
\label{LoSe}
\end{align}
for $\mu>0$.
The original Logvenenko-Sereda inequality with unknown dependence of $\delta, \mu$ was shown in \cite{LoSe74}. The sharp constant  with explicit dependence of $\delta, \mu$ as in (\ref{LoSe}) was obtained in \cite{K}.  The proof of (\ref{LoSe}) relies on analyticity of $g$ and analytic estimates (i.e. Bernstein-type inequality).

To show this connection between spectral inequalities (\ref{spectral-conti}) and uncertainty inequalities  (\ref{LoSe}), let us introduce some notations in spectral geometry. We denote $dP_{\lambda}$  the spectral measure of the operator $H=-\triangle +V(x)$. Using  spectral analysis, see e.g. \cite{RS81},
 we have
\begin{align*}
f=\int_{-\infty}^\infty dP_{\lambda}f,
\end{align*}
and
\begin{align*}
H (f)=\int_{-\infty}^\infty \lambda dP_{\lambda}f.
\end{align*}
For continuous function $F$ and $G$, we have
\begin{align*}
 F(H) f=\int_{-\infty}^\infty F(\lambda) dP_{\lambda}f  \quad  \mbox{and} \quad  (F(H) f, G(H) f)=\int_{-\infty}^\infty F(\lambda)G(\lambda) d(P_{\lambda}f, f).
\end{align*}
Then the spectral projector for the operator $H$ is given as
\begin{align}
\mathbb{I}_{\mu}(f)=\mathds{{1}}_{H\leq \mu}=\int^{\mu}_{-\infty} dP_{\lambda}f.
\label{proj-mu}
\end{align}

We consider  the Laplace operator in $\mathbb R^n$. Notice that the spectrum $\sigma(-\triangle)=[0, \infty)$ is absolutely continuous. Then \begin{align*}
\mathbb{I}_{\mu}(f)=\mathds{{1}}_{-\triangle\leq \mu}f=\int^{\mu}_{-\infty} dP_{\lambda}f
\end{align*}
for any $\mu>0$.
Using Fourier transform, we can rewrite it as
\begin{align*}
\mathds{{1}}_{-\triangle\leq \mu}f=\int_{\mathbb R^n} \hat{f}(\xi)\chi_{\mathbb B_{\mu}}(\xi) e^{2\pi i x\cdot\xi}\, d\xi,
\end{align*}
where
\begin{align*}
\hat{f}(\xi)=\int_{\mathbb R^n }f(x) e^{-2\pi x\cdot \xi} \, d\xi \quad \quad \mbox{for} \,  f\in L^2 (\mathbb R^n).
\end{align*}
Let $V(x)\equiv 0$ in Theorem \ref{th2}. One can deduce the equivalence of the spectral inequalities (\ref{spectral-conti}) and  the uncertainty inequalities  (\ref{LoSe}) by choosing $g=\mathds{{1}}_{-\triangle\leq \mu}f$. Such a equivalence between spectral inequalities  (\ref{spectral-conti}) and uncertainty inequalities  (\ref{LoSe}) for the Laplace operator $-\triangle$ has been observed and studied earlier in e.g. \cite{MV18}, \cite{WWZZ19}. It is known the thick sets in (\ref{geom-om-2}) are necessary for the validity of the uncertainty inequalities  (\ref{LoSe}). Thus, (\ref{geom-om-2}) seems to be necessary for the spectral inequality (\ref{spectral-conti}) in Theorem \ref{th2}.

The spectral inequalities imply the  observability inequality for heat equations. The observability inequality asserts that the  total energy of solutions can be estimated from above by the energy localized in a subdomain with an observability constant.
We study the following heat equation
\begin{equation}
    \left\{
    \begin{aligned}
        u_t-\triangle u +V(x)u &= 0 \quad  &\text{in }& \ \mathbb R^n \times (0, T) , \\
        u(\cdot,0) &= u_0    &\text{on }& \  \mathbb R^n .
    \end{aligned}
    \right.
    \label{nonlinear-one}
\end{equation}
 We will deduce the  observability inequality from  (\ref{aim-res}) in Theorem \ref{th1}. The observation region is restricted over the product of a subset of positive measure in time and $\Omega $ satisfying (\ref{geom-om}). The observablity inequality for the observation region on sets of positive measure in time or in a bounded domain has been studied in e.g. \cite{PW13} and \cite{AEWZ14}.  Assume that  the Assumption (A) holds for $V(x)$. Then  $\|V^-\|_{L^\infty}\leq C^0$, where $V^-(x)=\max\{ -V(x), 0\}$. We can show the following observablity inequality.

\begin{theorem}
Let $J\subset (0, T)$ be a measurable set of positive measure, $\Omega$ be in (\ref{geom-om}) and  $\sigma_2= \frac{1}{2}-\frac{\sigma}{\beta_1}>0$. Then any solution $u(x,t)$ of (\ref{nonlinear-one}) satisfies
\begin{align}
\| u(x, T)\|_{L^2(\mathbb R^n)}\leq  C(J) e^{C_1T\|V^-\|_{\infty}} e^{C(J) (\ln \frac 1 \delta )^{\frac{1}{ \sigma_2}}}
\|u\|_{L^2(\Omega\times J)},
\label{obser-1}
\end{align}
where $C(J)$ depends on $J$ and the constants in Assumption (A), and $C_1$ depends on the constants in Assumption (A).
\label{th3}
\end{theorem}

The observability inequalities imply the null controllability of heat equations. The heat equation 
\begin{equation}
    \left\{
    \begin{aligned}
        u_t-\triangle u +V(x)u &= f(x, t)\mathds{{1}}_{\Omega\times J} \quad  &\text{in }& \ \mathbb R^n \times (0, T) , \\
        u(\cdot,0) &= u_0    &\text{on }& \  \mathbb R^n
    \end{aligned}
    \right.
    \label{null-con}
\end{equation}
is said to be null controllable from the set $\Omega$ in any given time $T$ if, for any initial data $u_0\in L^2(\mathbb R^n)$, there exists a control function $f\in L^2(\mathbb R^n\times (0, T))$ supported in ${\Omega\times J}$ such that the solution of (\ref{null-con}) satisfies $u(x,T)=0$. By Hilbert uniqueness methods (see e.g. Theorem 2.44 in \cite{C07}), under the assumptions of Theorem \ref{th3}, the observability inequality (\ref{obser-1}) shows that the heat equation (\ref{null-con}) is null controllable. The spectral inequality (\ref{spectral-conti}) in Theorem \ref{th2} also shows an observability inequality for the heat equation (\ref{null-con}) for $V(x)$ under the Assumption (B) and (\ref{geom-om-2}). Since it's proof is similar to the proof of Theorem \ref{th3}, we do not pursue the argument here.

The organization of the paper is as follows. In section 2, we obtain the proof of Theorem \ref{th1}. We also show a quantitative result on propagation of smallness for gradients in Lemma \ref{lemm-5}.  Section 3 is devoted to the proof of Theorem
 \ref{th2}. In section 3, we discuss the applications of spectral inequalities to observability inequalities.  The letters $C$, ${C}^{i}$, $\hat{C}^i$, $C_i$ denote positive constants that do not depend on $\lambda$ or $\mu$, and may vary from line to line. 

\begin{remark}
    Several related works became available simultaneously when the first version of the paper was published in arXiv. In \cite{W24}, some non-sharp spectral inequality was studied for (\ref{eigen-k}) with $n=1$ and polynomial type growth potential $V(x)$ using quasiconformal mappings.  Right after our paper, the authors in \cite{LM24} obtained the similar results as Theorem \ref{th2} using  a slightly different method from a different perspective. Compared with \cite{LM24}, our paper further quantitatively characterizes the density of sensor sets in the spectral inequality in  Theorem \ref{th2}.
\end{remark}

\textbf{Acknowledgements.} The  author is partially supported by NSF DMS-2154506.

\section{Proof of the first type spectral inequality }
In this section, we will show the proof of the spectral inequality (\ref{aim-res}) in Theorem \ref{th1}.
We first make use of the lifting argument (or so-called ghost dimension construction) to get rid of  $\lambda_k$ in (\ref{eigen-k}). Since $V(x)$ is bounded below and grows to infinity, all the eigenvalues $\lambda_k$ are bounded below, say, $\lambda_k\geq -E_0$ for some $E_0\geq 0$. Let $\phi \in Ran(P_\lambda(H))$ be given by (\ref{phi-s}) with eigenpairs $(\phi_k,\lambda_k)$ satisfying (\ref{eigen-k}).
We introduce
\begin{equation}
\mathcal{S}_{\lambda_k}(s)=\left \{
\begin{array}{lll}
\frac{\sinh(\sqrt{\lambda_k}s)}{\sqrt{\lambda_k}}, \quad &\lambda_k>0, \nonumber \\
s, &\lambda_k=0, \nonumber \\
\frac{\sinh (i\sqrt{-\lambda_k}s)}{i\sqrt{-\lambda_k}}, \quad &\lambda_k<0.
\end{array}
\right.
\end{equation}
We construct
\begin{align}
\hat{\Phi}(x, s)=\sum_{-\infty<\lambda_k\leq \lambda} \alpha_k \phi_k(x) \mathcal{S}_{\lambda_k}(s).
\label{phi-c}
\end{align}
Then $\hat{\Phi}(x, s)$ satisfies the equation
\begin{align}
-{\triangle}\hat{\Phi} +V(x) \hat{\Phi}=0 \quad \mbox{in} \ \mathbb R^{n+1}.
\label{PPhi}
\end{align}
Note that $D_{s} \hat{\Phi} (x, 0)=\phi(x)$ and $\hat{\Phi}(x, 0)=0$, where $\triangle \hat{\Phi} =\sum^{n}_{i=1} D^2_i \hat{\Phi}+D^2_{ss}\hat{\Phi}$ in (\ref{PPhi}). Also notice that the Laplace operator $\triangle$ and the gradient operator $\nabla$ may be taken on different variables as we do lifting arguments later on. For convenience, we will still use the notation $\triangle$ and $\nabla$ if the context is understood.

We need the decay estimates for linear combination of eigenfunctions $\phi$ at infinity. 
The following lemma quantifies the decay property of $\phi$, which has been studied in, e.g. \cite{GY} and \cite{DSV}. Interested readers may refer to Theorem 1.4 in \cite{DSV} for a detailed proof.
\begin{lemma}
There exists a constant $\hat{C}$, depending on $\beta_1$, $c_1$, $C^0$ and $c_2$ such that for all $\lambda \ge 1$ and $\phi\in Ran(P_\lambda(H))$, we have
\begin{align}
\|\phi\|^2_{H^1(\mathbb R^n\backslash \mathcal B_{\frac{1}{2}\hat{C}\lambda^{1/\beta_1}})} \leq \frac{1}{2} \|\phi\|^2_{L^2(\mathbb R^n)}.
\label{decay-cr}
\end{align}
\label{propo-new}
\end{lemma}

We can compare the $L^2$ norm of $\phi$ and $H^1$ norm of $\hat{\Phi}$. The estimate is standard. We refer the interested readers to consult e.g. \cite{JL} or \cite{ZZ23} for a detailed proof. The readers may also check the proof of Lemma \ref{lem-com} for the proof in the same spirit. 
\begin{lemma}
Let $\phi\in Ran (P_\lambda(H))$ and $\hat{\Phi}$ be given in (\ref{phi-c}). For any $\lambda\ge 1$ and small $\rho>0$, we have
\begin{align}
2\rho \|\phi\|^2_{L^2(\mathbb R^n)}\leq \|\hat{\Phi}\|^2_{H^1(\mathbb R^{n}\times (-\rho, \rho))}\leq 2\rho(1+\frac{\rho^2}{3}(1+\lambda) ) e^{2\rho \sqrt{\lambda}} \|\phi\|^2_{L^2(\mathbb R^n)}.
\label{com-phi}
\end{align}
\label{lemma-1}
\end{lemma}

Relied on the decay estimates in the last lemma, we are able to show some  doubling type estimates for $\hat{\Phi}$.
\begin{lemma}
Let $\hat{\Phi}$ be in (\ref{PPhi}). For any $\lambda\geq 1$ and  $\rho>0$, we have
\begin{align}
\|\hat{\Phi}\|_{H^1 (\mathbb R^n\times (-\frac{\rho}{2}, \frac{\rho}{2}))}^2 \leq C(1+\lambda)\|\hat{\Phi}\|_{L^2 (\mathcal B_{\hat{C}\lambda^{1/\beta_1}}\times (-{\rho}, {\rho}))}^2
\label{crucial-cut}
\end{align}
and 
\begin{align}
\|\hat{\Phi}\|^2_{H^1(\mathbb R^{n}\times (-{4\rho}, {4\rho}))}\leq C e^{9\rho \sqrt{\lambda}}\| \hat{\Phi}\|_{H^1(\mathcal B_{\frac{1}{2}\hat{C}\lambda^{1/\beta_1}}\times (-\frac{\rho}{2}, \frac{\rho}{2}))}^2.
\label{kao-niu}
\end{align}
\label{lem-2}
\end{lemma}

\begin{proof}

Based on (\ref{decay-cr}), we can show that the global $H^1$ norm of $\Phi$ can be controlled by its local norm.
In fact, it follows from (\ref{decay-cr}) that
\begin{align*}
\|\phi\|^2_{H^1(\mathbb R^n\backslash \mathcal B_{\frac{1}{2}\hat{C}\lambda^{1/\beta_1}})}\leq  \|\phi\|^2_{L^2(\mathcal B_{\frac{1}{2}\hat{C}\lambda^{1/\beta_1}})}.
\end{align*}
This yields that
\begin{align}
\|\phi\|^2_{H^1(\mathbb R^n)}\leq 2 \|\phi\|^2_{H^1(\mathcal B_{\frac{1}{2}\hat{C}\lambda^{1/\beta_1}})}
\label{cut-off}
\end{align}
and
\begin{align}
\|\phi\|^2_{L^2 (\mathbb R^n)}\leq 2 \|\phi\|^2_{L^2(\mathcal B_{\frac{1}{2}\hat{C}\lambda^{1/\beta_1 }})}.
\label{L2}
\end{align}

Since $\hat{\Phi}(\cdot, s)\in Ran(P_\lambda(H))$, the estimate (\ref{cut-off}) implies that
\begin{align}
\|\hat{\Phi}\|^2_{H^1(\mathbb R^n)}\leq 2 \|\hat{\Phi}\|^2_{H^1(\mathcal B_{\frac{1}{2}\hat{C}\lambda^{1/\beta_1}})}.
\label{inte-x}
\end{align}
Since $D_{s} \hat{\Phi}(\cdot, s)\in Ran(P_\lambda(H))$ as well, the estimate (\ref{L2}) shows that
\begin{align}\label{inte-xn+1}
\|D_{s} \hat{\Phi}\|^2_{L^2 (\mathbb R^n)}\leq 2 \|D_{s} \hat{\Phi}\|^2_{L^2(\mathcal B_{\frac{1}{2}\hat{C}\lambda^{1/\beta_1}})}.
\end{align}
Combining both \eqref{inte-x} and \eqref{inte-xn+1} and integrating in $s$ over $(-\frac{\rho}{2}, \frac{\rho}{2})$, we obtain that
\begin{align}
\|\hat{\Phi}\|_{H^1 (\mathbb R^n\times (-\frac{\rho}{2}, \frac{\rho}{2}))}^2 \leq 2 \|\hat{\Phi}\|_{H^1(\mathcal B_{\frac{1}{2}\hat{C}\lambda^{1/\beta_1}}\times (-\frac{\rho}{2}, \frac{\rho}{2}))}^2.
\label{cut-est}
\end{align}

As $\hat{\Phi}$ satisfies  the elliptic equation (\ref{PPhi}), we apply the following Caccioppoli inequality
\begin{align}
 \| D \hat{\Phi}\|_{L^2(\mathcal{B}_{1}(z)\times (-\frac{\rho}{2},\frac{\rho}{2} ))}^2 & \le C  \| \hat{\Phi} \|_{L^2(\mathcal{B}_{2}(z)\times (-{\rho},{\rho})}^2 + C \| |V|^{\frac12} \hat{\Phi}\|_{L^2(\mathcal{B}_{2}(z)\times (-{\rho},{\rho}))}^2.
 \label{cappo}
\end{align}
 We cover $\mathcal B_{\frac{1}{2}\hat{C}\lambda^{1/\beta_1}}\times (-\frac{\rho}{2}, \frac{\rho}{2})$ by a finite number of $\mathcal{B}_1(z_i)\times (-\frac{\rho}{2}, \frac{\rho}{2})$ with finite overlaps.  The union of these balls also satisfies 
 \begin{align*}
 \cup_{i} \mathcal{B}_2(z_i)\times (-{\rho}, {\rho})\subset \mathcal{B}_{\hat{C}\lambda^{1/\beta_1}}\times (-{\rho}, {\rho} ).
 \end{align*}
 Because of the finite overlaps, using (\ref{cappo}), we have
\begin{align*}
 \| D \hat{\Phi}\|_{L^2( \mathcal{B}_{\frac{1}{2}\hat{C}\lambda^{1/\beta_1}}\times (-\frac{\rho}{2},\frac{\rho}{2} )}^2 \leq C(1+\lambda)
 \| \hat{\Phi}\|_{L^2( \mathcal{B}_{\hat{C}\lambda^{1/\beta_1}}\times (-{\rho}, {\rho} ))}^2,
\end{align*}
where we used $\|V\|_{L^\infty(\mathcal{B}_{\hat{C}\lambda^{1/{\beta_1}}})}\leq C(1+\lambda)$ which is from in Assumption (A).
It follows from (\ref{cut-est}) that
\begin{align}
\|\hat{\Phi}\|_{H^1 (\mathbb R^n\times (-\frac{\rho}{2}, \frac{\rho}{2}))}^2 \leq C(1+\lambda)\|\hat{\Phi}\|_{L^2 (\mathcal B_{\hat{C}\lambda^{1/\beta_1}}\times (-{\rho}, {\rho}))}^2.
\end{align}
This completes the proof of (\ref{crucial-cut}).
%\end{proof}

%We establish a doubling inequality for $\Phi$.
%\begin{lemma}
%Let $\Phi$ be given in (\ref{PHHI}). For any $\lambda\ge 1$ and small $\rho>0$, we have
%\begin{align}
%\|\nabla {\Phi}\|_{L^2(\mathcal{B}_{4\hat{C}\lambda^{1/\beta_1}}\times ({-4\rho}, {4\rho} )\times ({-4\rho}, {4\rho} )\times %({-4\rho}, {4\rho} ) )}^2 \leq C(1+\lambda) e^{6\rho \sqrt{\lambda}}
% \| \nabla {\Phi}\|_{L^2(\mathcal{B}_{\hat{C}\lambda^{1/\beta_1}}\times (-{\rho}, {\rho} )\times (-{\rho},{\rho} )\times ({-\rho}, {\rho} ))}^2.
% \label{double-phi}
%\end{align}
%\end{lemma}

%\begin{proof}
Thanks to (\ref{com-phi}), we have
\begin{align}
\rho \|\phi\|^2_{L^2(\mathbb R^n)}\leq \|\hat{\Phi}\|^2_{H^1(\mathbb R^{n}\times (-\frac{\rho}{2}, \frac{\rho}{2}))}\leq \rho(1+\frac{\rho^2}{12}(1+\lambda) ) e^{2\rho \sqrt{\lambda}} \|\phi\|^2_{L^2(\mathbb R^n)}
\label{camp-1}
\end{align}
and
\begin{align}
8\rho \|\phi\|^2_{L^2(\mathbb R^n)}\leq \|\hat{\Phi}\|^2_{H^1(\mathbb R^{n}\times (-{4\rho}, {4\rho}))}\leq 8\rho(1+\frac{16\rho^2}{3}(1+\lambda) ) e^{8\rho \sqrt{\lambda}} \|\phi\|^2_{L^2(\mathbb R^n)}.
\label{camp-2}
\end{align}
The combination of (\ref{camp-1}) and (\ref{camp-2}) yields that
\begin{align}
\|\hat{\Phi}\|^2_{H^1(\mathbb R^{n}\times (-{4\rho}, {4\rho}))}\leq 8(1+\frac{16\rho^2}{3}(1+\lambda) ) e^{8\rho \sqrt{\lambda}}
\|\hat{\Phi}\|^2_{H^1(\mathbb R^{n}\times (-\frac{\rho}{2}, \frac{\rho}{2}))}.
\label{half-cut}
\end{align}
Then it follows from (\ref{cut-est}) and (\ref{half-cut})  that
\begin{align}
\|\hat{\Phi}\|^2_{H^1(\mathbb R^{n}\times (-{4\rho}, {4\rho}))}\leq C e^{9\rho \sqrt{\lambda}}\| \hat{\Phi}\|_{H^1(\mathcal B_{\frac{1}{2}\hat{C}\lambda^{1/\beta_1}}\times (-\frac{\rho}{2}, \frac{\rho}{2}))}^2.
\label{kao-niu-1}
\end{align}
Therefore, the lemma is arrived.
\end{proof}

Next we apply the propagation of smallness results for gradients in \cite{LM18} to derive a refined three-ball type inequality with an explicit exponent. For a second order uniformly elliptic equation  
\begin{align}
{\rm div}(  A(\bar x)\nabla W)=0  
\label{elliptic}
\end{align}
in a bounded domain $\Omega_0 \subset \mathbb R^d$,  we assume that
\begin{align}
M^{-1}_1|\xi|^2 \leq  \langle A\xi, \xi \rangle\leq M_1 |\xi|^2,
\label{metric-1}
\end{align}
and
\begin{align}
|a_{ik}(\bar x_1)-a_{ik}(\bar x_2)|\leq M_2|\bar x_1-\bar x_2|,
\label{metric}
\end{align}
where $A=(a_{ik}(\bar x))_{d\times d}$ is a positive definite symmetric matrix and  $M_1, M_2$ are positive constants.
We define the doubling index for a non-trivial solution $W$ as
\begin{align*}
N(W, \tilde{B})=\log \frac{\sup_{2\tilde{B}}|W|}{\sup_{\tilde{B}}|W|},      
\end{align*}
where $\tilde{B}=\tilde{B}_r(\bar x)\subset \mathbb R^d$ is a ball centered at $\bar x$ with radius $r$,  $m\tilde{B}$ is the ball with the same center as $\tilde{B}$ and $m$ times the radius of $\tilde{B}$. The same notation applies for the cubes, e.g. $m Q^d$, in the later presentation.
We write $N( \tilde{B})$ instead of $N(W, \tilde{B})$ if the context is understood.
For the second order elliptic equation, the doubling index is almost monotonic
\begin{align*}
N(t \tilde{B})\leq N( \tilde{B})(1+c)+C,
\end{align*}
where $c$ and $C$ depend only on ${A}$, and $0<t\leq\frac{1}{2}$.
The Hausdorff content of a measurable set $E$ is defined as
\begin{align*}
\mathcal{C}^d(E)=\inf\{ \sum_{j} r_j^d: \  E\subset \cup \tilde{B}_{ r_j}(x_j)\}.
\end{align*}

We denote by $|E|$ the Lebesgue measure of the set $E$. We recall the Hausdorff content of order $d$ is comparable with the Lebesgue measure in the $d$-dimensional Euclidean space. That is,
\begin{align}
c_d|E|\leq \mathcal{C}^d(E)\leq C_d |E| \quad \mbox{for some } \ C_d, c_d>0.
\end{align}

To indicate the dimension in the estimates, we denote $Q^{d}$ as a $d$-dimensional cube with size $2$ centered at the origin. The following three-ball type inequality is given in $L^2$ norm with an explicit exponent.
\begin{lemma}
Let $W$ be the solution of (\ref{elliptic}) in $4Q^{n+3}$ satisfying (\ref{metric-1}), (\ref{metric}) with $\bar x=(x, s, t, y)$. There exist $C_1$ and $C_2$ depending on $M_1$, $M_2$, $n$ and $0<\gamma<1$ such that
\begin{align}
\|\nabla W\|_{L^\infty ( Q^{n+3})}\leq (\frac{2}{|E|})^{\frac{\gamma}{2}} \|\nabla W\|^\gamma _{L^2(E)} \|\nabla W \|^{1-\gamma}_{L^\infty( 2Q^{n+3})},
\label{logunov-mal-2-1}
\end{align}
where $\gamma=\frac{1} {\frac{1}{C_2} \ln \frac{ C_1|Q^{n+2}|  }{|E|}+1}$ and the measurable set $E\subset\frac{1}{2}Q^{n+3}\cap \{s=0\}$.
\label{lemm-5}
\end{lemma}

\begin{proof}
We apply the propagation  of smallness of the gradient results, i.e.  Lemma 5.3 in \cite{LM18}, to have
\begin{align}
\mathcal{C}^{n+2}(E)\leq C_1\big( \frac{\sup_{E}|\nabla W|}{\sup_{Q^{n+3}}|\nabla W|}\big)^{\frac{C_2} {N(\nabla W, {Q^{n+3}})}}\mathcal{C}^{n+2}(Q^{n+3}),
\label{logu-mal}
\end{align}
where $C_1, C_2$ depend on $M_1, M_2$ and $n$.
A direct consequence is the following three-ball type inequality,
\begin{align}
\sup_{Q^{n+3}}|\nabla W|\leq C\sup_{E}|\nabla W|( \frac{C_1 \mathcal{C}^{n+2}(Q^{n+3})}{ \mathcal{C}^{n+2}(E)})^{\frac{N(\nabla W, {Q^{n+3}})}{C_2}}.
\label{logu-mal-1}
\end{align}
Since  Hausdorff content of order $n+2$ is comparable with $n+2$ dimensional Lebesgue measure, we may identify $|Q^{n+2}|=\mathcal{C}^{n+2}(Q^{n+3})$.
From (\ref{logu-mal}),  we have
\begin{align*}
|\{(x, 0, t, y)\in Q^{n+3} | \ |\nabla W|< (\frac{\eps }{C_1 |Q^{n+2}|})^{\frac{1}{C_2}\ln \frac{\|\nabla W\|_{L^\infty(2 Q^{n+3})}}{ \|\nabla W\|_{L^\infty( Q^{n+3})}}} \|\nabla W\|_{L^\infty(Q^{n+3})} \}| <\eps
\end{align*}
for any $\eps>0$.
Let $E\subset\frac{1}{2}Q^{n+3}\cap\{s=0\}$ and $\eps=\frac{|E|}{2}$.
We have
\begin{align*}
|\{(x, 0, t, y)\in Q^{n+3} | \ |\nabla W|< (\frac{|E| }{2 C_1 |Q^{n+2}|})^{\frac{1}{C_2}\ln \frac{\|\nabla W\|_{L^\infty(2 Q^{n+3})}}{ \|\nabla W\|_{L^\infty( Q^{n+3})}}} \|\nabla W\|_{L^\infty(Q^{n+3})} \}| <\frac{|E|}{2}.
\end{align*}
Then
\begin{align*}
\int_{E}|\nabla W|^2 &\geq \int_{E} \chi_{|\nabla W|\geq (\frac{|E| }{2 C_1 |Q^{n+2}|})^{\frac{1}{C_2}\ln \frac{\|\nabla W\|_{L^\infty(2 Q^{n+3})}}{ \|\nabla W\|_{L^\infty( Q^{n+3})}}} \|\nabla W\|_{L^\infty(Q^{n+3})}} |\nabla W|^2 \nonumber \\
&\geq \frac{|E| }{2} (\frac{|E| }{2 C_1 |Q^{n+2}|})^{\frac{2}{C_2}\ln \frac{\|\nabla W\|_{L^\infty(2 Q^{n+3})}}{ \|\nabla W\|^2_{L^\infty( Q^{n+3})}}} \|\nabla W\|^2_{L^\infty(Q^{n+3})}.
\end{align*}
Therefore, we have
\begin{align*}
\frac{\|\nabla W\|_{L^2(E)}} {\|\nabla W\|_{L^\infty( Q^{n+3})} }\geq  (\frac{|E| }{2C_1|Q^{n+2}|})^{\frac{1}{C_2}\ln \frac{\|\nabla W\|_{L^\infty(2 Q^{n+3})}}{ \|\nabla W\|_{L^\infty( Q^{n+3})}} +\frac{1}{2}} (C_1|Q^{n+2}|)^{\frac{1}{2}}.
\end{align*}

Taking logarithms to both sides of  the last inequality, we arrive at
\begin{align*}
\ln \frac{\|\nabla W\|_{L^\infty( Q^{n+3})}} { \|\nabla W\|_{L^2(E)} }\leq \ln \frac{ C_1|Q^{n+2}|  }{|E|} \frac{1}{C_2}\ln \frac{ \|\nabla W\|_{L^\infty(2 Q^{n+3})}}{ \|\nabla W\|_{L^\infty( Q^{n+3})}}+{\frac{1}{2}} \ln \frac{2}{|E|}.
\end{align*}
That is,
\begin{align}
\frac{\|\nabla W\|_{L^\infty( Q^{n+3})}} { \|\nabla W\|_{L^2(E)}}\leq (\frac{\|\nabla W\|_{L^\infty(2 Q^{n+3})}}{ \|\nabla W\|_{L^\infty( Q^{n+3})}})^{\frac{1}{C_2} \ln \frac{ C_1|Q^{n+2}|  }{|E|}} (\frac{2}{|E|})^{\frac{1}{2}}.
\label{gamma-i}
\end{align}
We obtain that
\begin{align*}
\|\nabla W\|^{\frac{1}{C_2} \ln \frac{ C_1|Q^{n+2}|  }{|E|}+1}_{L^\infty( Q^{n+3})}\leq  \|\nabla W\|_{L^2(E)}( \|\nabla W\|)^{\frac{1}{C_2} \ln \frac{ C_1|Q^{n+2}|  }{|E|}}_{L^\infty( 2Q^{n+3})}(\frac{2}{|E|})^{\frac{1}{2}}.
\end{align*}
Therefore, we obtain the inequality
\begin{align*}
\|\nabla W\|_{L^\infty ( Q^{n+3})}&\leq (\frac{2}{|E|})^{\frac{\gamma}{2}}\|\nabla W\|^\gamma _{L^2(E)} \|\nabla W\|^{1-\gamma}_{L^2( 2Q^{n+3})}, \nonumber\\
%& \leq C \|\nabla W\|^\gamma _{L^2(E)} \|\nabla W\|^{1-\gamma}_{L^\infty( 2Q^{n+3})} ,
\end{align*}
where $\gamma=\frac{1} {\frac{1}{C_2} \ln \frac{ C_1|Q^{n+2}|  }{|E|}+1}$. This completes the proof of the lemma.
%and $C$ depends on $M_1$, $M_2$, $n$, and $diam(Q^{n+3})$.
\end{proof}

\begin{remark}
Compared with the three-ball type inequality for gradients (i.e. Theorem 5.1) in \cite{LM18}, which was derived from Lemma 5.3, we quantitatively show how $\gamma$ depends on the measure of the sets $E$  with positive measure. The explicit form of $\gamma$ is essential in showing how the spectral inequalities rely on the density of the positive measure sets.
\end{remark}

\begin{remark}
We can decrease the value of $\gamma$ in Lemma \ref{lemm-5}. For example, 
from (\ref{gamma-i}), it holds that
\begin{align}
\frac{\|\nabla W\|_{L^\infty( Q^{n+3})}} { \|\nabla W\|_{L^2(E)}}\leq (\frac{\|\nabla W\|_{L^\infty(2 Q^{n+3})}}{ \|\nabla W\|_{L^\infty( Q^{n+3})}})^{\frac{1}{C_2} \ln \frac{ 2C_1|Q^{n+2}|  }{|E|}}(\frac{2}{|E|})^{\frac{1}{2}}.
\end{align}
Thus, we can choose $\gamma=\frac{1} {\frac{1}{C_2} \ln \frac{ 2C_1|Q^{n+2}|  }{|E|}+1}$ such that the inequality (\ref{logunov-mal-2-1}) holds. This observation will help to choose a uniform $\gamma$ in the proof of Theorem \ref{th1}.
    \label{rem-3}
\end{remark}

We will perform two strategies to incorporate $V(x)$ into the leading coefficients of some second order elliptic equations without the lower order terms. Without loss of generality, we may assume $R=20$ in Assumption (A). We first consider the region away from $\mathcal{B}_{20}$

Let $\|V\|_{C^{0, 1}(\omega)}=\|V\|_{L^\infty(\omega)}+ \|DV\|_{L^\infty(\omega)}$ for a bounded domain $\omega\subset \mathbb R^n\backslash \mathcal{B}_{20},$ which will be determined later on. We want to incorporate $V(x)$ into the leading coefficients of a second order elliptic equation. 
%See \cite{LZ22} for another strategy in getting rid of a possible large parameter. 
As the potential $V(x)\geq -C^0$. We choose some $C_0>0$ such that $V(x)\geq -C^0\geq -C_0+1$. Recall $\hat{\Phi}$ satisfies (\ref{PPhi}).
We  work on the equation
\begin{align*}
-{\triangle}\hat{\Phi} +(V(x)+5C_0+4\|V\|_{C^{0, 1}(\omega)}) \hat{\Phi}=(5C_0+4\|V\|_{C^{0, 1}(\omega)} )\hat{\Phi} \quad \mbox{in} \ \mathbb R^{n+1}.
\end{align*}
Let $$\tilde{{\Phi}}(x, s, t)= e^{i2{(\|V\|_{C^{0, 1}(\omega)}+C_0)^{\frac{1}{2}}} t } \hat{\Phi}=\sum_{-\infty<\lambda_k\leq \lambda} \alpha_k \phi_k(x) \mathcal{S}_{\lambda_k}(s)
e^{i2{(\|V\|_{C^{0, 1}(\omega)}+C_0)^{\frac{1}{2}}} t }.$$ Then
\begin{align*}
-{\triangle}\tilde{\Phi} -\frac{V(x)+5C_0+4\|V\|_{C^{0, 1}(\omega)}}{4(\|V\|_{C^{0, 1}(\omega)}+C_0)}  \partial^2_{tt}\tilde{\Phi}=(5C_0+4\|V\|_{C^{0, 1}(\omega)})\tilde{\Phi} \quad \mbox{in} \ \mathbb R^{n+2}.
\end{align*}
For conveniences of the presentation, let us introduce $\tau_0=\sqrt{5C_0+4\|V\|_{C^{0, 1}(\omega)}}$.
Furthermore, we choose
\begin{align}
{\Phi}(x, s, t, y)&= e^{\tau_0 y}\tilde{{\Phi}} \nonumber\\
&=\sum_{-\infty<\lambda_k\leq \lambda} \alpha_k \phi_k(x) \mathcal{S}_{\lambda_k}(s)
e^{i2{(\|V\|_{C^{0, 1}(\omega)}+C_0)^{\frac{1}{2}}} t } e^{\tau_0 y}.
\label{kao-1}
\end{align} 
Then
\begin{align*}
-{\triangle}{\Phi} -\frac{V(x)+5C_0+4\|V\|_{C^{0, 1}(\omega)}}{4(\|V\|_{C^{0, 1}(\omega)}+C_0)}  D^2_{tt} {\Phi}- D_{yy}^2 {\Phi}=0 \quad \mbox{in} \ \mathbb R^{n+3}.
\end{align*}
Therefore, we can write it as
\begin{align}
-{\rm div} \big ({A}(x, s, t, y)\nabla \Phi\big)=0,
\label{PHHI}
\end{align}
where
\begin{align}
{A}(x, s, t,y)=\begin{bmatrix}
    I_{n\times n}       & 0 & 0 & 0\\
 0     & 1 &  0 & 0\\
 0       & 0 & \frac{V(x)+5C_0+4\|V\|_{C^{0, 1}(\omega)} }{4({\|V\|_{C^{0, 1}(\omega)}}+C_0)}& 0 \\
 0 & 0& 0& 1
 \label{b-matrix}
\end{bmatrix}.
\end{align}

%We construct
%\begin{align}
%{\Phi} (x, s, t)= e^{i{(\|V\|_{C^{0, 1}(\omega)}+1)^{\frac{1}{2}}} t } \hat{\Phi}=\sum_{0<\lambda_k\leq \lambda} \alpha_k %\phi_k(x) \frac{\sinh(\sqrt{\lambda_k}s) }{\sqrt{\lambda_k}}e^{i{(\|V\|_{C^{0, 1}(\omega)}+1)^{\frac{1}{2}}} t }.
%\label{con
%Then
%\begin{align}
%-\triangle {\Phi}-\frac{V(x)}{(\|V\|_{C^{0, 1}(\omega)}+1)} D^2_{tt}{\Phi}=0 \quad \mbox{in} \  \omega\times \mathbb R\times \mathbb R.
%\label{Vequ}
%\end{align}
% (Attention, we need to $\|V\|_{\rm {Lip}}$ is infinity in $R^{n+2}$. We need to restrict to $B_\lambda$.)

%Furthermore, we can write (\ref{Vequ}) in the following way
%\begin{align}
%-{\rm div} \big (A(x, s, t)\nabla \Phi\big)=0,
%\label{PHII}
%\end{align}
%where
%\begin{align}
%A(x, s, t)=\begin{bmatrix}
%    I_{n\times n}       & 0 & 0 \\
% 0     & 1 &  0 \\
% 0       & 0 & \frac{V(x)}{({\|V\|_{C^{0, 1}(\omega)}}+1)}
%\end{bmatrix}.
%\end{align}
Since $V(x)$ is local Lipschitz continuous, it is easy to see that $A(x, s, t, y)$ is uniformly elliptic and local Lipschitz continuous in $ \mathbb R^{n+3}$.
We will consider the quantitative properties for $\nabla \Phi$. Direct calculations show that
\begin{align}
\nabla \Phi(x, s, t,y)=\langle &\sum_{-\infty<\lambda_k\leq \lambda} \alpha_k \nabla \phi_k(x) \mathcal{S}_{\lambda_k}(s) e^{i2(\|V\|_{C^{0, 1}(\omega)}+C_0)^{\frac{1}{2}} t} e^{\tau_0 y},  \nonumber \\ &\sum_{-\infty<\lambda_k\leq \lambda} \alpha_k  \phi_k(x) \partial_s \mathcal{S}_{\lambda_k}(s) e^{i2(\|V\|_{C^{0, 1}(\omega)}+C_0)^{\frac{1}{2}} t }e^{\tau_0 y}, \nonumber \\
& i2{(\|V\|_{C^{0, 1}(\omega)}
+C_0)^{\frac{1}{2}}}\sum_{-\infty<\lambda_k\leq \lambda} \alpha_k \phi_k(x) \mathcal{S}_{\lambda_k}(s)
e^{i2(\|V\|_{C^{0, 1}(\omega)}+C_0)^{\frac{1}{2}}t}e^{\tau_0 y},\nonumber \\
&\tau_0 \sum_{-\infty<\lambda_k\leq \lambda} \alpha_k  \phi_k(x)  \mathcal{S}_{\lambda_k}(s) e^{i(\|V\|_{C^{0, 1}(\omega)}+C_0)^{\frac{1}{2}} t }e^{\tau_0 y}\rangle.
\label{PPhi-1}
\end{align}

%Recall that ${\Phi}= e^{i{(\|V\|_{C^{0, 1}(\omega)}+C_0)^{\frac{1}{2}}} t } e^{\sqrt{C_0}y} \hat{\Phi}$. Hence
% $$\nabla \Phi(x, s, t, y)=\langle e^{i{(\|V\|_{C^{0, 1}(\omega)}+C_0)^{\frac{1}{2}}} t } e^{\sqrt{C_0}y} \nabla_{(x, s)}\hat{\Phi}, i{(\|V\|_{C^{0, 1}(\omega)}+C_0)^{\frac{1}{2}}}{\Phi}, \sqrt{C_0}{\Phi} \rangle.$$
Let $Q_j=\Lambda_2(j)$ for $j\in \mathbb Z^n$. Then we have
 \begin{align}
 \|\nabla \Phi\|^2_{L^2(4Q_j\times (-{4\rho}, {4\rho})\times (-{4\rho}, {4\rho})\times (-{4\rho}, {4\rho}))}
 &\leq 128 \rho^2 (5C_0+ 8\|V\|_{C^{0,1}(\omega)})e^{8\tau_0 \rho} \|\hat{\Phi}\|^2_{L^2(Q_j \times (-{4\rho}, {4\rho}))}
 \nonumber \\ & +128\rho^2e^{8\tau_0 \rho}\|\nabla \hat{\Phi}\|^2_{L^2( 4Q_j\times (-{4\rho}, {4\rho}))}\nonumber \\
&\leq C\rho^2(5C_0+ 8\|V\|_{C^{0,1}(\omega)}) e^{8\tau_0\rho}\|\hat{\Phi}\|^2_{H^1 ( 4Q_j\times (-{4\rho}, {4\rho}))}.
\label{recall-1}
 \end{align}
Similarly, we can obtain that
 \begin{align}
 \|\nabla \Phi\|^2_{L^2(Q_j\times (-{\rho}, {\rho})\times (-{\rho}, {\rho})\times (-{\rho}, {\rho}))}
&\geq C\rho^2 e^{-2\tau_0\rho}\|\hat{\Phi}\|^2_{H^1 ( Q_j\times (-{\rho}, {\rho}))}.
\label{recall-1-1}
 \end{align}

 Next we consider the region in $\mathcal{B}_{20}$. Notice that $V(x)$ is only bounded in the region. We aim to build a new function to get rid of $V(x)$ and construct a new second order elliptic equation without the lower order terms. As \begin{align*}
-{\triangle}\hat{\Phi} +(V(x)+C_0) \hat{\Phi}=C_0 \hat{\Phi} \quad \mbox{in} \ \mathbb R^{n+1}.
\end{align*}
We introduce $\hat{\Phi}_1= e^{\sqrt{C_0} t}\hat{\Phi} $. Then $\hat{\Phi}_1$ satisfies 
\begin{align}
-{\triangle}\hat{\Phi}_1 +(V(x)+C_0) \hat{\Phi}_1=0 \quad \mbox{in} \ \mathbb R^{n+2}.
\label{new-construction}
\end{align}

Due to (\ref{new-construction}), we consider the existence of the solution $w(x,s, t)$ for the following equation 
\begin{equation}
    \begin{aligned}
        -\triangle w +(V(x)+C_0)w = 0  \quad  \text{in } \ \mathcal{B}_{30}\times (-2, 2)\times (-2, 2).
        \label{mistake-1}
    \end{aligned}
\end{equation}
%By maximum principle, $w(x,s) >0$ in $ \mathcal{B}_{30}\times (-2, 2)$. Let us normalize $\|w\|_{L^1({ \mathcal{B}_{20}\times (-2, 2)})}=1$.  Applying the weak Harnack inequality in \cite{GT77} (Theorem 8.18), we have
%\begin{align}
%\inf_{\mathcal{B}_{25}\times (-\frac{3}{2}, \frac{3}{2})} w\geq C \|w\|_{L^1({\mathcal{B}_{30}\times (-2, 2)})}\geq C.
%\label{harnack}
%\end{align}
It is true that $1<V(x)+C_0\leq 2 C_0$ in $\mathcal{B}_{30}\times (-2, 2)\times (-2, 2)$. On one hand,  choosing $w_1=e^{\sqrt{2C_0}(x_1+\cdots+x_n+s+t)}$, then
\begin{align*}
    -\triangle w_1+(V(x)+C_0) w_1\leq 0 \ \quad  \text{in } \ \mathcal{B}_{30}\times (-2, 2)\times (-2, 2).
\end{align*}
On the other hand, letting $w_2=e^{40 n \sqrt{2C_0}}$, it holds that
\begin{align*}
    -\triangle w_2+(V(x)+C_0) w_2\geq 0 \ \quad  \text{in } \ \mathcal{B}_{30}\times (-2, 2)\times (-2, 2),
\end{align*}
since $(V(x)+C_0)>0$.
Note that  $w_1\leq w_2$. By the sub-solution and super-solution method, there exists a solution $w$ satisfying (\ref{mistake-1}) and
\begin{align}
0< e^{-40n \sqrt{2{C}_0} }\leq w_1\leq w\leq w_2 \leq e^{40 n \sqrt{2{C}_0}}.
\label{harnack}
\end{align}
We introduce a new function $\bar \Phi(x, s)=\frac{\hat{\Phi}_1}{w}$, where $\hat{\Phi}_1$ is given in (\ref{new-construction}). Then $\bar \Phi$ satisfies the equation 
\begin{align}
-{\rm{div}}(w^2 \nabla \bar \Phi)=0    \quad  \text{in } \ \mathcal{B}_{30}\times (-2, 2)\times (-2, 2).
\end{align}
By (\ref{harnack}) and standard elliptic estimates for solutions $w$ in (\ref{mistake-1}), we have
\begin{align}
    C^{-1}\leq \| w^2\|_{L^\infty({\mathcal{B}_{25}\times (-\frac{3}{2}, \frac{3}{2})\times (-\frac{3}{2}, \frac{3}{2})})} \leq C
    \label{elli-1}
\end{align}
and
\begin{align}
    \| w\|_{C^{0,1}(\mathcal{B}_{20}\times (-1, 1)\times (-1, 1))}\leq C \| w\|_{L^\infty({\mathcal{B}_{25}\times (-\frac{3}{2}, \frac{3}{2})\times (-\frac{3}{2}, \frac{3}{2})})}\leq C.
        \label{elli-2}
\end{align}
Thus,
\begin{align}
    \| w^2\|_{C^{0,1}(\mathcal{B}_{20}\times (-1, 1)\times (-1, 1))}\leq C \| w\|^2_{L^\infty({\mathcal{B}_{25}\times (-\frac{3}{2}, \frac{3}{2})\times (-\frac{3}{2}, \frac{3}{2})})}\leq C.
    \label{elli-3}
\end{align}

We consider the equation 
\begin{align*}
-{\rm{div}}(w^2 \nabla \bar \Phi)-\hat{C}_0 \bar \Phi+\hat{C}_0 \bar \Phi=0   
\end{align*}
for some large constant $\hat{C}_0>0$, which to be determined and is used to control the norm of $w$. Let $\tilde{\Phi}= e^{\sqrt{\hat{C}_0} y}\bar{\Phi}$. Then
\begin{align*}
-{\rm{div}}(w^2 \nabla \tilde{\Phi})-\partial_{yy}\tilde{\Phi}+ {\hat{C}_0} \tilde{\Phi}=0  \quad \text{in} \ \mathcal{B}_{30}\times (-2, 2)\times (-2, 2)\times \mathbb{R}.
\end{align*}
Furthermore, choose ${\Phi}_1(x, s, t, y,\tau )= e^{i \sqrt{{\hat{C}_0}}  \tau}\tilde{\Phi}.$
Then
\begin{align*}
    -{\rm{div}}(w^2 \nabla {\Phi}_1)-\partial_{yy}{\Phi}_1-\partial_{\tau\tau} {\Phi}_1=0 \quad \text{in} \ \mathcal{B}_{30}\times (-2, 2)\times (-2, 2)\times \mathbb{R}\times \mathbb{R}.
\end{align*}
We can write the last equation as
\begin{align}
 -{\rm{div}}(A(x,s, t, y,\tau)\nabla {\Phi}_1)=0, 
 \label{second-elli}
\end{align}
where
\begin{align}
{A}(x, s, t,y,\tau)=\begin{bmatrix}
    w^2(x, s,t) I_{n\times n}       & 0 & 0 & 0& 0\\
 0     & w^2(x, s, t) &  0 & 0 & 0\\
  0     & 0 & w^2(x, s, t) &  0 & 0 \\
 0    & 0   & 0 & 1& 0 \\
 0 & 0 & 0& 0& 1
\end{bmatrix}.
\label{b-matrix-1}
\end{align}

Note that
\begin{align*}
    \Phi_1(x,s, t,y,\tau)&=  e^{\sqrt{{C}_0} t} e^{\sqrt{\hat{C}_0} y} e^{i\sqrt{\hat{C}_0} \tau} \hat{\Phi} w^{-1}\nonumber \\&= e^{\sqrt{{C}_0} t} e^{\sqrt{\hat{C}_0} y} e^{i\sqrt{\hat{C}_0} \tau} w^{-1} \sum_{-\infty<\lambda_k\leq \lambda} \alpha_k \phi_k(x) \mathcal{S}_{\lambda_k}(s).
\end{align*}

Direct calculations show that
\begin{align*}
\nabla \Phi_1=&\langle  e^{\sqrt{{C}_0} t} e^{\sqrt{\hat{C}_0} y} e^{i\sqrt{\hat{C}_0} \tau} (\nabla_x \hat{\Phi} w^{-1}-
\hat{\Phi} w^{-2}\nabla_x w), \nonumber \\
 &  e^{\sqrt{{C}_0} t} e^{\sqrt{\hat{C}_0} y} e^{i\sqrt{\hat{C}_0} \tau}(\partial_s \hat{\Phi} w^{-1}-\hat{\Phi} w^{-2} \partial_s w), \nonumber  \\
  &  e^{\sqrt{{C}_0} t} e^{\sqrt{\hat{C}_0} y} e^{i\sqrt{\hat{C}_0} \tau}\hat{\Phi} (\sqrt{{C}_0} w^{-1}- w^{-2} \partial_t w), \nonumber  \\
&  \sqrt{\hat{C}_0} e^{\sqrt{{C}_0} t} e^{\sqrt{\hat{C}_0} y} e^{i\sqrt{\hat{C}_0} \tau} \hat{\Phi} w^{-1}, \nonumber  \\ 
&i\sqrt{\hat{C}_0} e^{\sqrt{{C}_0} t} e^{\sqrt{\hat{C}_0} y} e^{i\sqrt{\hat{C}_0} \tau} \hat{\Phi} w^{-1}\rangle.
\end{align*}
Let $0<\rho<1$. Using estimates for  $w$ and $\nabla w$ in (\ref{elli-1}) and (\ref{elli-2}), we have
\begin{align*}
&\|\nabla \Phi_1\|^2_{L^2(Q_j\times (-\rho, \rho)\times (-\rho, \rho) \times (-\rho, \rho)\times (-\rho, \rho))} \nonumber \\
&\geq C\| e^{\sqrt{{C}_0} t} e^{\sqrt{\hat{C}_0} y} w^{-1}\nabla \hat{\Phi} \|^2_{L^2(Q_j\times (-\rho, \rho)\times (-\rho, \rho) \times (-\rho, \rho)\times (-\rho, \rho))} \nonumber \\ &-C_1 \| e^{\sqrt{{C}_0} t} e^{\sqrt{\hat{C}_0} y}  \hat{\Phi} w^{-2}\nabla w\|_{L^2(Q_j\times (-\rho, \rho)\times (-\rho, \rho) \times (-\rho, \rho)\times (-\rho, \rho))} \nonumber \\
&+\hat{C}_0 \| e^{\sqrt{{C}_0} t} e^{\sqrt{\hat{C}_0} y}   \hat{\Phi} w^{-1}\|^2_{L^2(Q_j\times (-\rho, \rho)\times (-\rho, \rho) \times (-\rho, \rho)\times (-\rho, \rho))} \nonumber \\
& \geq C\|  e^{\sqrt{{C}_0} t} e^{\sqrt{\hat{C}_0} y}  \nabla \hat{\Phi} \|^2_{L^2(Q_j\times (-\rho, \rho)\times (-\rho, \rho) \times (-\rho, \rho)\times (-\rho, \rho))} \nonumber \\& + C\|  e^{\sqrt{{C}_0} t} e^{\sqrt{\hat{C}_0} y}  \hat{\Phi} \|^2_{L^2(Q_j\times (-\rho, \rho)\times (-\rho, \rho) \times (-\rho, \rho)\times (-\rho, \rho))}
\end{align*}
by choosing $\hat{C}_0$ large enough. Thus, we have 
\begin{align}
\|\nabla \Phi_1\|^2_{L^2(Q_j\times (-\rho, \rho)\times (-\rho, \rho) \times (-\rho, \rho)\times (-\rho, \rho))}
    \geq C_1 \| \hat{\Phi} \|^2_{H^1(Q_j\times (-\rho, \rho))}.
    \label{brain-1}
\end{align}
Similarly, we can also verify that 
\begin{align}
\|\nabla \Phi_1\|^2_{L^2(Q_j\times (-\rho, \rho)\times (-\rho, \rho) \times (-\rho, \rho)\times (-\rho, \rho))}
    \leq C_2 \| \hat{\Phi} \|^2_{H^1(Q_j\times (-\rho, \rho))}.
    \label{brain-2}
\end{align}

We will apply the Lemma \ref{lemm-5} to derive some quantitative three-ball type results for $\Phi$ in the equation (\ref{PHHI}) and $\Phi_1$ in the equation (\ref{second-elli}).
We will consider $\nabla \Phi$ on a half space $\{(x, 0, t, y)\}\cap \mathbb R^{n+3}$. We introduce 
\begin{align} \hat{\phi}(x, t, y)= e^{i2{(\|V\|_{C^{0, 1}(\omega)}+C_0)^{\frac{1}{2}}} t } e^{\tau_0 y} \phi.
\label{why-1}
\end{align}
Note that \begin{align} \nabla \Phi(x, 0,t, y)=\langle 0, \hat{\phi}(x,t, y) , 0, 0 \rangle. \label{dean} \end{align}
Similarly we will consider $\nabla \Phi_1$ on a half space $\{(x, 0, t, y,\tau)\}\cap \mathcal{B}_{30}\times (-2, 2)\times (-2, 2)\times \mathbb{R}\times \mathbb{R} $.
Let \begin{align}
    \hat{\phi}_1(x,0,t,y,\tau)= e^{\sqrt{{C}_0} t} e^{\sqrt{\hat{C}_0} y} e^{i\sqrt{\hat{C}_0} \tau}
\phi w^{-1}.
\label{why-2}\end{align}
It is true that
$$\nabla \Phi_1(x,0,t,y,\tau)=\langle 0, \hat{\phi}_1, 0, 0, 0\rangle.$$ 
Recall that $Q^{n+4}$ denotes the $n+4$ dimensional cube. Let the measurable set $E_1\subset \frac 1 2 Q^{n+4}\cap\{s=0\}.$
\begin{corollary}
 Assume $\Phi$ satisfies the equation (\ref{PHHI}), $\Phi_1$ satisfies the equation (\ref{second-elli}) and their leading coefficients (\ref{b-matrix}), (\ref{b-matrix-1}) satisfies (\ref{metric-1}) and (\ref{metric}).
Let $\hat{\phi}$ be given in (\ref{why-1}) and $\hat{\phi}_1$ be given in (\ref{why-2}). Then
  \begin{align}
\|\nabla \Phi\|_{L^2 ( Q^{n+3})}\leq C (\frac{|Q^{n+4}|}{|E|})^{\frac{\gamma}{2}} \|\hat{\phi}\|^\gamma _{L^2(E)} \|\nabla \Phi\|^{1-\gamma}_{L^2( 4Q^{n+3})},
\label{logunov-mal-2}
\end{align}  
and 
  \begin{align}
\|\nabla \Phi_1\|_{L^2 ( Q^{n+4})}\leq C (\frac{|Q^{n+4}|}{|E|})^{\frac{\gamma}{2}} \|\hat{\phi}_1\|^\gamma _{L^2(E_1)} \|\nabla \Phi_1\|^{1-\gamma}_{L^2( 4Q^{n+4})},
\label{logunov-mal-2-3}
\end{align}
\label{cor-1-1}
where $\gamma=\frac{1} {\frac{1}{C_2} \ln \frac{ C_1|Q^{n+3}|  }{|E_1|}+1}$ if $|E_1|=|E|.$
\end{corollary}
\begin{proof}

From the explicit expression of $\nabla \Phi$ in (\ref{PPhi-1}), and the elliptic estimates for $\Phi$ in (\ref{PHHI}),
we get
\begin{align}
\|\nabla\Phi\|_{L^\infty (2Q^{n+3})}\leq \frac{C}{|Q^{n+3}|^{\frac{1}{2}}}  \|\Phi\|_{L^2 (4Q^{n+3})}\leq  \frac{C}{|Q^{n+3}|^{\frac{1}{2}}}\|\nabla\Phi\|_{L^2 (4Q^{n+3})}.
\label{cacc-2}
\end{align}
Hence, it follows from (\ref{logunov-mal-2-1}) and (\ref{dean}) that 
\begin{align}
\|\nabla \Phi\|_{L^2 ( Q^{n+3})}\leq C (\frac{|Q^{n+3}|}{|E|})^{\frac{\gamma_1}{2}}  \|\hat{\phi}\|^{\gamma_1}_{L^2(E)} \|\nabla \Phi\|^{1-\gamma_1}_{L^2( 4Q^{n+3})},
\label{logunov-mal-3}
\end{align}
where $\gamma_1=\frac{1} {\frac{1}{C_2} \ln \frac{ C_1|Q^{n+2}|  }{|E|}+1}$, and $C$ depends on $M_1$, $M_2$, $n$.
By the same arguments, we can show that
  \begin{align*}
\|\nabla \Phi_1\|_{L^2 ( Q^{n+4})}\leq C (\frac{|Q^{n+4}|}{|E|})^{\frac{\gamma}{2}}   \|\hat{\phi}_1\|^\gamma _{L^2(E_1)} \|\nabla \Phi_1\|^{1-\gamma}_{L^2( 4Q^{n+4})},
\end{align*}
where $\gamma=\frac{1} {\frac{1}{C_2} \ln \frac{ C_1|Q^{n+3}|  }{|E_1|}+1}$. Note that $|Q^{n+2}|=2^{n+2}$ and  $|Q^{n+3}|=2^{n+3}$.
Thanks to Remark \ref{rem-3}, we can change $\gamma_1$ to be $\gamma$ if $|E_1|=|E|.$
\end{proof}
%%\begin{align*}
%\|\nabla \Phi\|^{\frac{1}{C_2} \ln \frac{ C_1|Q^{n+2}|  }{|E|}+1}_{L^2 ( Q^{n+3})}\leq C(C_1|Q^{n+2}|)^{\frac{1}{2}} \|\hat{\phi}\|_{L^2(E)} \|\nabla \Phi\|^{\frac{1}{C_2}\ln \frac{ C_1|Q^{n+2}|  }{|E|}}_{L^2( 4Q^{n+3})}
%\end{align*}
%Therefore, we obtain the inequality
%\begin{align*}
%\|\nabla \Phi\|_{L^2 ( Q^{n+3})}&\leq (C_1|Q^{n+2}|)^{\frac{1}{2\gamma}}\|\hat{\phi}\|^\gamma _{L^2(E)} \|\nabla \Phi\|^{1-\gamma}_{L^2( 4Q^{n+3})} \nonumber\\
%& \leq C \|\hat{\phi}\|^\gamma _{L^2(E)} \|\nabla \Phi\|^{1-\gamma}_{L^2( 4Q^{n+3})} ,
%\end{align*}
% C depends on $1/r$ and r is size of Q.
Now we are ready to give the proof of Theorem \ref{th1}.

%From (\ref{double-2}), we have
%\begin{align}
%N(\nabla \Phi, \mathcal{Q}_j)\leq C\sqrt{\lambda}.
%\end{align}

\begin{proof}[Proof of Theorem \ref{th1}]
We will apply the quantitative three-ball type inequality (\ref{logunov-mal-2-3}) and (\ref{logunov-mal-2}). We first consider a region near origin.
We introduce $\Omega_j=\Omega\cap \Lambda_1(j)$ and $Q_j=\Lambda_2(j)$.
Let $\mathcal{E}_j=\Omega_j\times \{0\}\times (-\frac{\rho}{2}, \frac{\rho}{2} )\times (-\frac{\rho}{2}, \frac{\rho}{2} ) \times (-\frac{\rho}{2}, \frac{\rho}{2} )$ and
$\mathcal{Q}_j=Q_j\times (-{\rho},{\rho} )\times  (-{\rho},{\rho} )\times  (-{\rho},{\rho} )\times  (-{\rho},{\rho} )$. By rescaling estimates, e.g. by considering $\Phi_1(x, \frac{s}{\rho}, \frac{t}{\rho}, \frac{y}{\rho}, \frac{\tau}{\rho})$, we can identify the cubes $\mathcal{Q}_j$ with the standard cube $Q^{n+4}$ with size $2$,  and the measurable set $\mathcal{E}_j$ with the set $E_1$. By fixing the value of $\rho$, these rescalings only change  $A(x,s,t,y,
\tau)$  in (\ref{b-matrix-1}) by some constant. Thus, we apply the three-ball type inequality (\ref{logunov-mal-2-3}) in $\mathcal{Q}_j$ and $\mathcal{E}_j$.
To this end, we verify that the conditions  (\ref{metric-1}) and (\ref{metric}) holds for $A(x,s,t,y, \tau)$ in  (\ref{second-elli}). 
%We are going to apply the conclusion in Lemma \ref{lemm-5} near the origin. Recall that  $\Omega_j=\Omega\cap \Lambda_1(j)$ and $Q_j=\Lambda_2(j)$, $\mathcal{E}_j=\Omega_j\times \{0\}\times (-\frac{\rho}{2}, \frac{\rho}{2} )\times (-\frac{\rho}{2}, \frac{\rho}{2} ) $ and
From estimates (\ref{elli-1}), (\ref{elli-2}) and (\ref{elli-3}), we know that the leading coefficients $A(x, s,t, y, \tau)$ in (\ref{second-elli}) satisfies (\ref{metric-1}) and (\ref{metric}) for some constants $M_1$ and $M_2$. Thus,  the three-ball type inequality (\ref{logunov-mal-2-3}) in Corollary \ref{cor-1-1} holds for $\Phi_1$,
\begin{align}
\| \nabla \Phi_1\|_{L^2 ( \mathcal{Q}_j)}\leq C (\frac{|\mathcal{Q}_j|}{|\mathcal{E}_j|})^{\frac{\gamma}{2}}  \|\hat{\phi}_1\|^\gamma _{L^2(\mathcal{E}_j)} \|\nabla \Phi_1\|^{1-\gamma}_{L^2( 4\mathcal{Q}_j)},
\label{285}
\end{align}
where $\gamma =\frac{1}{\frac{1}{C_2}\ln \frac{C_1|\mathcal{Q}_j|^{\frac{n+3}{n+4}}}{|\mathcal{E}_j|}+1}.$ Note that $|Q^{n+3}|$ in Corollary \ref{cor-1-1} is the Lebesgue measure for the cube $Q^{n+3}$. Here $|\mathcal{Q}_j|$ denotes the Lebesgue measure for $\mathcal{Q}_j$ in $\mathbb R^{n+4}$.
%where $$\nabla \Phi_1(x,0,t,y)= \hat{\phi}_1(x,t,y)=e^{\sqrt{\hat{C}_0}t}e^{i\sqrt{\hat{C}_0}y}\phi w^{-1}$$ and $4\mathcal{Q}_j\subset \mathcal{B}_{20}\times (-4\rho, 4\rho)\times (-4\rho, 4\rho) \times (-4\rho, 4\rho)\times (-4\rho, 4\rho)$.

As it can be see from (\ref{brain-1}) and (\ref{brain-2}) that the $L^2$ norm of $\Phi_1$ and the $H^1$ norm of $\hat{\Phi}$ are comparable. Therefore, it follows from (\ref{285}) that
\begin{align}
\| \hat{\Phi}\|_{H^1 ({Q}_j\times (-\rho, \rho))}\leq C (\frac{|\mathcal{Q}_j|}{|\mathcal{E}_j|})^{\frac{\gamma}{2}}   \|{\phi}\|^\gamma _{L^2(\Omega_j)} \|\hat{\Phi}\|^{1-\gamma}_{H^1 (4{Q}_j\times (-4\rho, 4\rho))}
\label{287}
\end{align}   
for $4{Q}_j\subset \mathcal{B}_{20}$ since  $$\|\hat{\phi}_1 \| _{L^2(\mathcal{E}_j)} \leq C e^{\sqrt{\hat{C}_0}\rho}\|{\phi}\| _{L^2(\Omega_j)}.$$ 
%Since $V(x)=|x|^{\beta_1}$ satisfies the Assumption (A) for $x\in 4{Q}_j\subset \mathbb R^n \backslash \mathcal{B}_{20}$,
%it is true that (\ref{287}) holds for $4{Q}_j\subset \mathbb R^n \backslash \mathcal{B}_{20}$
%Now we know that (\ref{287}) holds near the origin and away from the origin. 

%Let $E=\mathcal{E}_j$ and $Q^{n+3}=\mathcal{Q}_j$. 
%By the assumption (A), $\|V\|_{C^{0, 1}(\mathcal{B}_{20})}\leq C$. We can set $\omega=\mathcal{B}_{20}$.
%If $4Q^{n+3}\subset\mathcal{B}_{20}\times (-4\rho, 4\rho)\times (-4\rho, 4\rho)\times (-4\rho, 4\rho)$, the conditions (\ref{metric-1}) and (\ref{metric}) are satisfied for $A(x,s,t,y)$ in (\ref{b-matrix}).  Then (\ref{logunov-mal-2}) is applicable. We have
%\begin{align}
%\| \nabla \Phi\|_{L^2 ( \mathcal{Q}_j)}\leq C \|\hat{\phi}\|^\gamma _{L^2(\mathcal{E}_j)} \|\nabla \Phi\|^{1-\gamma}_{L^2( 4\mathcal{Q}_j)}.
%\end{align}

Next we consider the region away from $\mathcal{B}_{20}$. Let $\mathcal{E}^1_j=\Omega_j\times \{0\}\times (-\frac{\rho}{2}, \frac{\rho}{2} )\times (-\frac{\rho}{2}, \frac{\rho}{2} ) \times (-\frac{\rho}{2}, \frac{\rho}{2} )$ and
$\mathcal{Q}^1_j=Q_j\times (-{\rho},{\rho} )\times  (-{\rho},{\rho} )\times  (-{\rho},{\rho} )\times  (-{\rho},{\rho} )$.  By rescaling estimates,  we still identify the cubes $\mathcal{Q}^1_j$ with the standard cube $Q^{n+3}$ with size $2$,  the measurable set $\mathcal{E}^1_j$ with the set $E$ in Corollary \ref{cor-1-1}.
Let $\omega=\Lambda_{10}(j)$ for $|j|\geq 7$. If $4\mathcal{Q}^1_j \subset(\mathcal{B}_{\hat{C}\lambda^{\frac{1}{\beta_1}}}\backslash\mathcal{B}_{20})\times (-4\rho, 4\rho)\times (-4\rho, 4\rho)\times (-4\rho, 4\rho)$, by the Assumption (A), we get
\begin{align*}
   \frac{3}{4}\leq \frac{V(x)+5C_0+4\|V\|_{C^{0, 1}(\omega)}}{4({\|V\|_{C^{0, 1}(\omega)}}+C_0)}\leq \frac{5}{4}
\end{align*}
and
\begin{align*}
  \frac{| \nabla V(x)|}{4({\|V\|_{C^{0, 1}(\omega)}}+C_0)}\leq \frac{1}{4}
\end{align*}
%\begin{align*}
%c_1 (|j|-6)^{\beta_1}-C^0\leq \|V\|_{C^{0, 1}(\omega)}\leq c_2(|j|+6)^{\beta_1}.
%\end{align*}
%\begin{align*}
% \frac{c_1}{4c_2}\leq  \frac{c_1(|j|-6)^{\beta_1}+1} {c_2(|j|+6)^{\beta_1}+C_0}\leq   \frac{V(x)+C_0}{\|V\|_{C^{0, 1}(\omega)} +C_0}\leq 1
%\end{align*}
for $x\in \omega$ and any $|j|$ large.
Thus, from (\ref{b-matrix}), the matrix $A$ in (\ref{b-matrix}) implies that the conditions (\ref{metric-1}) and (\ref{metric}) are satisfied for some $M_2$ and $M_2$ as well by changing the constants accordingly. From (\ref{logunov-mal-2}) in Corollary \ref{cor-1-1}, we have
\begin{align}
\| \nabla \Phi\|_{L^2 ( \mathcal{Q}^1_j)}\leq C (\frac{|\mathcal{Q}_j|}{|\mathcal{E}_j|})^{\frac{\gamma}{2}} \|\hat{\phi}\|^\gamma _{L^2(\mathcal{E}^1_j)} \|\nabla \Phi\|^{1-\gamma}_{L^2( 4\mathcal{Q}^1_j)}.
\label{sum-three}
\end{align}
Observe that
\begin{align}
    \|\hat{\phi}\|^2 _{L^2(\mathcal{E}^1_j)} \leq C\rho^2 e^{\tau_0\rho}\|{\phi}\|^2_{L^2 (\Omega_j)}.
    \label{recall-1-2}
\end{align}
It follows from (\ref{sum-three}), (\ref{recall-1}), (\ref{recall-1-1}) and (\ref{recall-1-2}) that
\begin{align*}
\|\hat{\Phi}\|_{H^1 ( Q_j\times (-{\rho}, {\rho}))}\leq C(5C_0+ 8\|V\|_{C^{0,1}(\omega)}) e^{10\tau_0\rho} (\frac{|\mathcal{Q}_j|}{|\mathcal{E}_j|})^{\frac{\gamma}{2}} \|{\phi}\|^\gamma_{L^2 (\Omega_j)} \|\hat{\Phi}\|^{1-\gamma}_{H^1 ( 4Q_j\times (-{4\rho}, {4\rho}))}.
\end{align*}
We have chosen $\omega =\Lambda_{10}(j)\subset \mathcal{B}_{\hat{C}\lambda^{1/\beta_1}}$. By the assumption (A),  it holds that
$$\|V\|_{C^{0,1}(\omega)}\leq C(1+\lambda)^{\frac{\beta_2}{\beta_1}}.$$ Thus, from the definition of $\tau_0$, we get
\begin{align}
\|\hat{\Phi}\|_{H^1 ( Q_j\times (-{\rho}, {\rho}))}\leq C(1+\lambda)^{\frac{\beta_2}{\beta_1}} e^{C \lambda^{\frac{\beta_2}{2\beta_1}}\rho} (\frac{|\mathcal{Q}_j|}{|\mathcal{E}_j|})^{\frac{\gamma}{2}} \|{\phi}\|^\gamma_{L^2 (\Omega_j)} \|\hat{\Phi}\|^{1-\gamma}_{H^1 ( 4Q_j\times (-{4\rho}, {4\rho}))}.
\label{sum-three-1}
\end{align}
Taking (\ref{287}) into considerations, the estimates (\ref{sum-three-1}) hold for any $4{Q}_j\times (-{4\rho}, {4\rho} ) \subset\mathcal{B}_{4\hat{C}\lambda^{1/\beta_1}}\times (-{4\rho}, {4\rho} )$. Note that 
\begin{align*}\frac {1}{\gamma}=\frac{1}{C_2}\ln \frac{C_1|Q^{n+3}|}{|{E}|}-1=\frac{1}{C_2}\ln \frac{C_1|\mathcal{Q}_j|^{\frac{n+3}{n+4}}}{|\mathcal{E}_j|}-1.
\end{align*}
From the definition of $\Omega_j$ in (\ref{geom-om}) and $|j|\leq C\lambda^{\frac{1}{\beta_1}}$, we get that
$$\frac{|\mathcal{Q}_j|}{|\mathcal{E}_j|}\leq (\frac{1}{\delta})^{1+\lambda^{\frac{\sigma}{\beta_1}}} $$
and
$$\frac {1}{\gamma}\leq \frac{1}{C_2} \ln C_1 (\frac{1}{\delta})^{1+\lambda^{\frac{\sigma}{\beta_1}}}.$$
We can choose 
\begin{align}
    \gamma= {C_2} \big(\ln C_1 (\frac{1}{\delta})^{1+\lambda^{\frac{\sigma}{\beta_1}}}\big)^{-1}
    \label{kao-1}
\end{align}
as (\ref{sum-three-1}) still holds.
We cover $ \mathcal{B}_{\hat{C}\lambda^{1/\beta_1}}\times (-{\rho}, {\rho} )$ by a union of ${Q}_j\times (-{\rho}, {\rho} )$ such that
\begin{align*}
\mathcal{B}_{\hat{C}\lambda^{1/\beta_1}}\times (-{\rho}, {\rho} ) \subset \cup_{j} {Q}_j\times (-{\rho}, {\rho} ).
\end{align*}

It is easy to see that
\begin{align*}
 \cup_{j} 4{Q}_j\times (-{4\rho}, {4\rho} ) \subset \mathcal{B}_{4\hat{C}\lambda^{1/\beta_1}}\times (-{4\rho}, {4\rho} ).
\end{align*}
The following H\"older's inequality holds
\begin{align}
\sum_j a_j^\gamma b_j^{1-\gamma}\leq (\sum_j a_j  )^\gamma (\sum_j b_j)^{1-\gamma}
\label{holder}
\end{align}
for $0<\gamma<1$ and $a_j, b_j>0$. 
Let 
\begin{align*}
a_j=\|\hat{\phi}\|_{L^2(\Omega_j)}, \quad  b_j=\|\hat{ \Phi}\|_{H^1( 4{Q}_j\times (-{4\rho}, {4\rho} ) )}.
\end{align*}
Taking the summation in (\ref{sum-three-1}) over a union of ${Q}_j\times (-{\rho}, {\rho} )$ yields that
\begin{align*}
\| \hat{\Phi} \|_{H^1 (\mathcal{B}_{\hat{C}\lambda^{1/\beta_1}}\times (-{\rho}, {\rho} ))}\leq &C(1+\lambda)^{\frac{\beta_2}{\beta_1}} e^{C \lambda^{\frac{\beta_2}{2\beta_1}}\rho} (\frac{1}{\delta})^{\frac{\gamma}{2}(1+\lambda^{\frac{\sigma}{\beta_1}})} \|\hat{\phi}\|^\gamma _{L^2(\Omega\cap \mathcal{B}_{4\hat{C}\lambda^{1/\beta_1}})} 
 \|\hat{\Phi}\|^{1-\gamma}_{H^1(  \mathcal{B}_{4\hat{C}\lambda^{1/\beta_1}}\times (-{4\rho}, {4\rho} ) )},
\end{align*}
where $ \Omega=\cup \Omega_j $.
Thanks to the doubling inequality (\ref{kao-niu}), we get
\begin{align*}
\|\hat{\Phi} \|_{H^1 (\mathcal{B}_{\hat{C}\lambda^{1/\beta_1}}\times (-{\rho}, {\rho} ))}\leq& C(1+\lambda)^{\frac{\beta_2}{\beta_1}} e^{C \lambda^{\frac{\beta_2}{2\beta_1}}\rho} e^{\sqrt{\lambda}(1-\gamma)}(\frac{1}{\delta})^{\frac{\gamma}{2}(1+\lambda^{\frac{\sigma}{\beta_1}})} \|{\phi}\|^\gamma _{L^2(\Omega)} \| \hat{\Phi} \|^{1-\gamma}_{H^1 (\mathcal{B}_{\hat{C}\lambda^{1/\beta_1}}\times (-{\rho}, {\rho} ))}.
\end{align*}
Therefore, from the fact that $\frac{\beta_2}{2\beta_1}\geq \frac{1}{2}$,  we have
\begin{align*}
\|\hat{\Phi} \|_{H^1 (\mathcal{B}_{\hat{C}\lambda^{1/\beta_1}}\times (-{\rho}, {\rho} ))}  &\leq C e^{\frac{C\lambda^{\frac{\beta_2}{2\beta_1}}}{\gamma}}
e^{\frac{C\sqrt{\lambda}(1-\gamma)}{\gamma}} (\frac{1}{\delta})^{\frac{1}{2}(1+\lambda^{\frac{\sigma}{\beta_1}})} \|{\phi}\|_{L^2(\Omega)}\nonumber \\
&\leq C e^{\frac{C\lambda^{\frac{\beta_2}{2\beta_1}}}{\gamma}}
(\frac{1}{\delta})^{C(1+\lambda^{\frac{\sigma}{\beta_1}})}\|{\phi}\|_{L^2(\Omega)}
\end{align*}
where $C$ depends on $\rho$, which is fixed.
It follows from (\ref{kao-1}) that
\begin{align*}
\|\hat{\Phi} \|_{H^1 (\mathcal{B}_{\hat{C}\lambda^{1/\beta_1}}\times (-{\rho}, {\rho} ))}  \leq C (\frac{1}{\delta})^{C \lambda^{\frac{\sigma}{\beta_1}+\frac{\beta_2}{2\beta_1}}}  \|{\phi}\|_{L^2(\Omega)}.
\end{align*}
%Recall that ${\Phi}= e^{i{(\|V\|_{C^{0, 1}(\omega)}+1)^{\frac{1}{2}}} t } e^{\sqrt{C_0}y} \hat{\Phi}$,  it holds that
%\begin{align}
%\| \nabla \Phi\|^2_{L^2 (\mathcal{B}_{\hat{C}\lambda^{1/\beta_1}}\times (-{\rho}, {\rho} ) \times (-{\rho}, {\rho} )  \times (-{\rho}, %{\rho} ))}\geq 4\rho^2 e^{-\sqrt{C_0}\rho} \| \hat{\Phi}\|^2_{L^2 (\mathcal{B}_{\hat{C}\lambda^{1/\beta_1}}\times (-{\rho}, {\rho} ) )}.
%\end{align}
%It is also true that
% $$\|\hat{\phi}\|_{L^2(\Omega\times (-\frac{\rho}{2}, \frac{\rho}{2})\times (-\frac{\rho}{2}, \frac{\rho}{2}))}\leq \rho^2 e^{\sqrt{C_0}\rho}\|{\phi}\|_{L^2(\Omega)} .$$ Thus, we obtain that
%\begin{align}
%\| \hat{\Phi}\|_{L^2 (\mathcal{B}_{\hat{C}\lambda^{1/\beta_1}}\times (-{\rho}, {\rho} ) )}\leq C (\frac{1}{\delta})^{C %\lambda^{\frac{\sigma}{\beta_1}+\frac{1}{2}}}
% \|{\phi}\|_{L^2(\Omega)}
%\end{align}
%for $\lambda\geq 1$.
Thanks to (\ref{crucial-cut}) in Lemma \ref{lem-2}, we have
\begin{align*}
\|\hat{\Phi}\|_{H^1 (\mathbb R^n\times (-\frac{\rho}{2}, \frac{\rho}{2}))} \leq C (\frac{1}{\delta})^{C \lambda^{\frac{\sigma}{\beta_1}+\frac{\beta_2}{2\beta_1}}}
 \|{\phi}\|_{L^2(\Omega)}.
\end{align*}
At last, applying the first inequality in (\ref{com-phi}) in Lemma \ref{lemma-1} and fixing the value of $\rho>0$, we arrive at
\begin{align*}
 \|\phi\|_{L^2(\mathbb R^n)}\leq C (\frac{1}{\delta})^{C \lambda^{\frac{\sigma}{\beta_1}+\frac{\beta_2}{2\beta_1}}}
 \|{\phi}\|_{L^2(\Omega)}.
\end{align*}
This completes the proof of Theorem \ref{th1}.
\end{proof}

\section{Proof of the second type spectral inequality }
This section is devoted to the proof of Theorem \ref{th2}. We will adapt the strategy in the proof of Theorem \ref{th1}. Using spectral measure $dP_\lambda$, we construct
\begin{align}
\check{f}=\int^\mu_{-\infty}\mathcal{S}_\lambda(s) dP_{\lambda}f,
\label{check-f}
\end{align}
where
\begin{equation}
\mathcal{S}_\lambda(s)=\left \{
\begin{array}{lll}
\frac{\sinh(\sqrt{\lambda}s)}{\sqrt{\lambda}}, \quad &\lambda>0, \nonumber \\
s, &\lambda=0, \nonumber \\
\frac{\sinh (i\sqrt{-\lambda}s)}{i\sqrt{-\lambda}}, \quad &\lambda<0.
\end{array}
\right.
\end{equation}
Then we obtain
\begin{equation*}
-\triangle \check{f}-\partial_{ss} \check{f}+V(x)\check{f}=0 \quad \quad \mbox{in} \ \mathbb{R}^{n+1}.
\end{equation*}
Since $\|V\|_{L^\infty}\leq C^0=C_0-1$ for some positive constant $C_0$, we have $1\leq C_0+V(x)\leq 2C_0$. We write the equations as 
\begin{equation*}
-\triangle \check{f}-\partial_{ss} \check{f}+(C_0+V(x))\check{f}-C_0 \check{f}=0 \quad \quad \mbox{in} \  \mathbb{R}^{n+1}.
\end{equation*}
Let $\tilde{f}=e^{\sqrt{C_0}t} \check{f}$.
Then
\begin{align*}
-\triangle \tilde{f}-\partial_{ss} \tilde{f} -\partial_{tt}\tilde{f} + (C_0+V(x)) \tilde{f}=0  \quad \quad  \mbox{in} \ \mathbb{R}^{n+2}.
\end{align*}

We construct  the solution  $w(x,s, t)$ in the following equation 
\begin{equation}
    \begin{aligned}
        -\triangle w +(C_0+V(x)) w &= 0  \quad  &\text{in }& \ \mathcal{B}_{30}(j)\times (-2, 2)\times (-2, 2),
        \label{mistake-1-1}
    \end{aligned}
\end{equation}
where $\mathcal{B}_{30}(j)$ is a ball centered at $j\in \mathbb Z^n$ with radius $30$. We will let $j$ change in the later proof.
Let $w_1= e^{\sqrt{2C_0}(x_1-j_1+ x_x-j_2\cdots+x_n-j_n+s+t)}$. It holds that
\begin{align*}
    -\triangle w_1+(C_0+V(x)) w_1\leq 0 \quad \text{in }\ \mathcal{B}_{30}(j)\times (-2, 2)\times (-2, 2).
\end{align*}
Let $w_2= e^{\sqrt{2C_0}}$. It is known that
\begin{align*}
    -\triangle w_2+(C_0+V(x)) w_2\geq 0 \quad \text{in }\ \mathcal{B}_{30}(j)\times (-2, 2)\times (-2, 2).
\end{align*}

Note that  $w_1\leq w_2$. By the sub-solution and super-solution method, there exists a solution $w$ satisfying (\ref{mistake-1-1}) and
\begin{align}
0< e^{-40n \sqrt{2C_0} }\leq w_1\leq w\leq w_2 \leq e^{40 n\sqrt{2C_0} }.
\label{harnack-1}
\end{align}

%By maximum principle, $w>0$ in $ \mathcal{B}_{30}(j)\times (-2, 2)\times (-2, 2)$. Let us normalize $\|w\|_{L^1({ \mathcal{B}_{30}(j)\times (-2, 2)\times (-2, 2)})}=1$. Since $V(x)\in L^\infty$, by elliptic regularity estimates, $w\in C^{1, \alpha}$ in $\mathcal{B}_{30}(j)\times (-2, 2)\times (-2, 2)$ for some $0<\alpha<1$. Applying the weak Harnack inequality in \cite{GT77} (Theorem 8.18), we have
%\begin{align}
%\inf_{\mathcal{B}_{25}(j)\times (-\frac{3}{2}, \frac{3}{2}) \times (-\frac{3}{2}, \frac{3}{2})} w\geq C \|w\|_{L^1({\mathcal{B}_{30}(j)\times (-2, 2)\times (-2, 2)})}\geq C.
%\label{harnack-1}
%\end{align}
We consider a new function $ f_1(x, s, t)=\frac{\tilde{f}}{w}$. Then $ f_1$ satisfies the equation 
\begin{align*}
-{\rm{div}}(w^2 \nabla f_1)=0    \quad  \text{in } \ \mathcal{B}_{30}(j)\times (-2, 2)\times (-2, 2).
\end{align*}
By (\ref{harnack-1}) and standard elliptic estimates for (\ref{mistake-1-1}),  it holds that
\begin{align}
    C^{-1}\leq \| w^2\|_{L^\infty({\mathcal{B}_{25}(j)\times (-\frac{3}{2}, \frac{3}{2})\times (-\frac{3}{2}, \frac{3}{2})})} \leq C
    \label{elli-1-1}
\end{align}
and 
\begin{align}
    \| w^2\|_{C^{0,1}(\mathcal{B}_{20}(j)\times (-1, 1)\times (-1, 1))}\leq C \| w\|^2_{L^\infty({\mathcal{B}_{25}(j\times (-\frac{3}{2}, \frac{3}{2})\times (-\frac{3}{2}, \frac{3}{2})})}\leq C.
    \label{elli-2-1}
\end{align}

To control the norm of $w$, we consider the equation 
\begin{align*}
-{\rm{div}}(w^2 \nabla f_1)-\hat{C}_0 f_1+\hat{C}_0 f_1=0   
\end{align*}
for some large constant $\hat{C}_0>0$, which to be determined. Let $f_2= e^{\sqrt{\hat{C}_0} y}{f}_1$. Then
\begin{align*}
-{\rm{div}}(w^2 \nabla f_2 )-\partial_{yy}f_2+ {\hat{C}_0} f_2=0  \quad \text{in} \ \mathcal{B}_{30}(j)\times (-2, 2)\times (-2, 2)\times \mathbb{R}.
\end{align*}

%\begin{align}
%-{\rm{div}}(w^2 \nabla \tilde{\Phi})-\partial_{yy}\tilde{\Phi}+ {\hat{C}_0} \tilde{\Phi}=0  \quad \text{in} \ \mathcal{B}_{20}\times (-2, %2)\times \mathbb{R}.
%\end{align}
Furthermore, denote $\bar f(x, s, t, y, \tau )= e^{i \sqrt{{\hat{C}_0}} \tau }f_2.$
We have
\begin{align*}
    -{\rm{div}}(w^2 \nabla \bar f)-\partial_{yy}\bar f-\partial_{\tau\tau} \bar f=0 \quad \text{in} \ \mathcal{B}_{30}(j)\times (-2, 2)\times (-2, 2)\times \mathbb{R}\times \mathbb{R}.
\end{align*}
The last equation can be written as
\begin{align}
 -{\rm{div}}(A(x,s, t, y,\tau)\nabla \bar f)=0, 
 \label{second-elli-1-1}
\end{align}
where
\begin{align}
{A}(x, s, t,y, \tau)=\begin{bmatrix}
    w^2(x, s, t)I_{n\times n}       & 0 & 0 & 0& 0\\
 0     & w^2(x, s, t) &  0 & 0& 0 \\
  0     & 0 &w^2(x, s, t) &  0 & 0\\
 0       & 0 & 0&1& 0 \\
 0 & 0& 0& 0&1
\end{bmatrix}.
\end{align}

%Let $\tilde{f}=e^{(\|V\|_{C^{0, 1}}+C_0)^{1/2}t} \check{f}$.
%Then
%\begin{align}
%-\triangle \tilde{f}-\partial_{ss} \tilde{f}-\frac{C_0-V(x)}{(\|V\|_{C^{0, 1}}+C_0)^{1/2} }\partial_{tt}\tilde{f}+C_0 \tilde{f}=0  \quad %\quad  \mbox{in} \ \mathbb{R}^{n+2}.
%\end{align}
%Finally, letting $\bar{f}= e^{i C_0^{1/2} y} \tilde{f}$ gives that
%\begin{align}
%-\triangle \bar{f}-\partial_{ss} \bar{f}-\frac{C_0-V(x)}{\|V\|_{C^{0, 1}}+C_0 }\partial_{tt} \bar{f}-\partial_{yy} \bar{f}=0  \quad \quad  %\mbox{in} \ \mathbb{R}^{n+3}.
%\end{align}
%Thus, we can write the last equation as
%\begin{align}
%-{\rm div} \big ({A}(x, s, t, y)\nabla \bar{f}\big)=0,
%\label{second-ell}
%\end{align}
%where
%\begin{align}
%{A}(x, s, t, y)=\begin{bmatrix}
%    I_{n\times n}       & 0 & 0 & 0\\
% 0     & 1 &  0 & 0\\
% 0       & 0 & \frac{C_0-V(x)}{{\|V\|_{C^{0, 1}}}+C_0}& 0 \\
% 0 & 0& 0& 1
%\end{bmatrix},
%\end{align}
It holds that
\begin{align}
\bar{f}(x, s, t, y,
\tau)=  w^{-1} e^{ \sqrt{C_0}  t} e^{ \sqrt{\hat{C_0} } y} e^{ i \sqrt{\hat{C_0} } \tau}\int^\mu_{-\infty}\mathcal{S}_\lambda(s) d P_{\lambda}f.
\label{bar-f}
\end{align}

From the construction of $\bar f$,  direct calculations show that
\begin{align}
\nabla \bar f(x, s, t, y,\tau)=&\langle e^{ \sqrt{C_0}  t} e^{ \sqrt{\hat{C_0} } y} e^{ i \sqrt{\hat{C_0} } \tau} (w^{-1}\nabla_x \int^\mu_{-\infty}\mathcal{S}_\lambda(s) d P_{\lambda}f- w^{-2} \nabla_x w \int^\mu_{-\infty}\mathcal{S}_\lambda(s) d P_{\lambda}f),  \nonumber\\
& e^{ \sqrt{C_0}  t} e^{ \sqrt{\hat{C_0} } y} e^{ i \sqrt{\hat{C_0} } \tau} (w^{-1}\partial_s \int^\mu_{-\infty}\mathcal{S}_\lambda(s) d P_{\lambda}f- w^{-2} \partial_s w \int^\mu_{-\infty}\mathcal{S}_\lambda(s) d P_{\lambda}f),  \nonumber  \\
&e^{ \sqrt{C_0}  t} e^{ \sqrt{\hat{C_0} } y} e^{ i \sqrt{\hat{C_0} } \tau} (w^{-1} \sqrt{C_0}- w^{-2} \partial_t w )\int^\mu_{-\infty}\mathcal{S}_\lambda(s) d P_{\lambda}f, \nonumber  \\
& \sqrt{\hat{C_0}}  \bar f,  \
 i \sqrt{\hat{C_0}} \bar f \rangle
\label{nabla-bar-f}
\end{align}
and 
\begin{align}
\nabla \bar f(x, 0, t, y, 
\tau) =& \langle0,  \ 
 w^{-1} e^{ \sqrt{C_0}  t} e^{ \sqrt{\hat{C_0} } y} e^{ i \sqrt{\hat{C_0} } \tau} \int^\mu_{-\infty}  d P_{\lambda}f, 0, \ 0, \ 0\rangle.
\end{align}
Let $0<\rho<1$. Recall that $Q_j=\Lambda_2(j)$. It is clear that $6Q_j\subset \mathcal{B}_{20}(j)$. The estimates for  $w$ and $\nabla w$ in (\ref{elli-1-1}) and (\ref{elli-2-1}) imply that
\begin{align*}
&\|\nabla \bar f \|^2_{L^2(Q_j\times (-\rho, \rho)\times (-\rho, \rho) \times (-\rho, \rho)\times (-\rho, \rho))} \nonumber\\
&\geq C\| e^{\sqrt{{C}_0} t} e^{\sqrt{\hat{C}_0} y} w^{-1} \nabla \int^\mu_{-\infty}\mathcal{S}_\lambda(s) d P_{\lambda}f   \|^2_{L^2(Q_j\times (-\rho, \rho)\times (-\rho, \rho) \times (-\rho, \rho)\times (-\rho, \rho))} \nonumber \\ &-C_1 \|e^{\sqrt{{C}_0} t} e^{\sqrt{\hat{C}_0} y}  \int^\mu_{-\infty}\mathcal{S}_\lambda(s) d P_{\lambda}f   w^{-2}\nabla w\|_{L^2(Q_j\times (-\rho, \rho)\times (-\rho, \rho) \times (-\rho, \rho) \times (-\rho, \rho)))} \nonumber \\
&+\hat{C}_0 \|e^{\sqrt{{C}_0} t} e^{\sqrt{\hat{C}_0} y}  \int^\mu_{-\infty}\mathcal{S}_\lambda(s) d P_{\lambda}f  
 w^{-1}\|^2_{L^2(Q_j\times (-\rho, \rho)\times (-\rho, \rho) \times (-\rho, \rho)\times (-\rho, \rho))} \nonumber \\
& \geq C\| e^{\sqrt{{C}_0} t} e^{\sqrt{\hat{C}_0} y} \nabla \int^\mu_{-\infty}\mathcal{S}_\lambda(s) d P_{\lambda}f\|^2_{L^2(Q_j\times (-\rho, \rho)\times (-\rho, \rho) \times (-\rho, \rho)\times (-\rho, \rho))} \nonumber \\& + C\| e^{\sqrt{{C}_0} t} e^{\sqrt{\hat{C}_0} y} \int^\mu_{-\infty}\mathcal{S}_\lambda(s) d P_{\lambda}f  \|^2_{L^2(Q_j\times (-\rho, \rho)\times (-\rho, \rho) \times (-\rho, \rho)\times (-\rho, \rho))}
\end{align*}
by choosing $\hat{C}_0$ large enough. Furthermore, we have 
\begin{align}
\|\nabla  \bar f\|^2_{L^2(Q_j\times (-\rho, \rho)\times (-\rho, \rho) \times (-\rho, \rho) \times (-\rho, \rho))}
    \geq C_1 \| \int^\mu_{-\infty}\mathcal{S}_\lambda(s) d P_{\lambda}f \|^2_{H^1(Q_j\times (-\rho, \rho))},
    \label{brain-1-1}
\end{align}
where $C_1$ depends on the fixed constant $\rho$ and the constants in Assumption (B).
The similar arguments also yield that
\begin{align}
\|\nabla  \bar f\|^2_{L^2(Q_j\times (-\rho, \rho)\times (-\rho, \rho) \times (-\rho, \rho))}
    \leq C_2 \| \int^\mu_{-\infty}\mathcal{S}_\lambda(s) d P_{\lambda}f \|^2_{H^1(Q_j\times (-\rho, \rho))}.
    \label{brain-2-1}
\end{align}

Next we show that the $L^2$ norm of the spectral projection $\mathbb{I}_{\mu}(f)$ and the $H^1$ norm of $\int^\mu_{-\infty}\mathcal{S}_\lambda(s) dP_{\lambda}f$ are comparable, which is stated in the following lemma.
\begin{lemma}
There exists a positive constant $C$ such that 
\begin{align}
C\rho\|\mathbb{I}_{\mu}(f)\|^2_{L^2(\mathbb R^n)} &\leq \|  \int^\mu_{-\infty}\mathcal{S}_\lambda(s) dP_{\lambda}f\|^2_{H^1(\mathbb R^n\times (-\rho, \rho))}\nonumber \\
&\leq C\rho(1+{\rho^2}(1+|\mu|) )
e^{2\rho (\sqrt{|\mu|}+1)} \|\mathbb{I}_{\mu}(f)\|^2_{L^2(\mathbb R^n)}
\label{compare-in}
\end{align}
for any small $\rho>0$.
\label{lem-com}
\end{lemma}

\begin{proof}
If $\mu<0$, the estimate (\ref{compare-in}) is clear from the definition of $\mathcal{S}_\lambda(s)$ and the fact that there are only a finite number of negative eigenvalue $\lambda$. See Lemma \ref{lemma-1} or the argument in \cite{ZZ23}. Let us focus on $\mu>0$
We first study the properties of $\cosh$ and $\sinh$. We have
\begin{align}
|\sinh (\sqrt{\lambda} s)|\leq |\sqrt{\lambda} s \cosh(\sqrt{\lambda}  s)|,
\label{calcu-1}
\end{align}
and
\begin{align}
1\leq \cosh( \sqrt{\lambda}  s)\leq e^{\sqrt{\lambda}  |s|}
\label{calcu-2}
\end{align}
for $\lambda\geq 0$. For $\lambda<0$, there are a finite number of eigenvalues $\lambda_i$ such that $\lambda_i>-C$ for some positive $C$ depending on $V(x)$. We choose some small positive  $\rho_0>0$ such that $\cos(\sqrt{-\lambda}\rho_0)=c_0$ for some positive constant $0<c_0<1$.
Thus, 
\begin{align*}
|\frac{\sinh(i\sqrt{-\lambda}s)}{i\sqrt{-\lambda}}|= |\frac{\sin(\sqrt{-\lambda}s)}{\sqrt{-\lambda}}|\leq |s|
\end{align*}
and 
\begin{align}
c_0< |\cosh(i\sqrt{-\lambda}s)|=|\cos(\sqrt{-\lambda}s)|<1
\label{315}
\end{align}
for $|s|<\rho\leq 100 \rho_0$.
Let us prove the first inequality in (\ref{compare-in}).  Note that there are only a finite number of eigenvalue $\lambda$ such that $\lambda<0$ due to the boundedness of $V(x)$. By the properties of $\cosh$, we get
\begin{align*}
\| \int^\mu_{-\infty} \partial_s \mathcal{S}_\lambda(s) dP_{\lambda}f\|^2_{L^2(\mathbb R^n\times (-\rho, \rho))}
&\geq \| \int^\mu_{0} \cosh(\sqrt{\lambda }s) dP_{\lambda}f\|^2_{L^2(\mathbb R^n\times (-\rho, \rho))}\nonumber \\
&+\| \int^0_{-\infty} \cosh(i\sqrt{-\lambda }s) dP_{\lambda}f\|^2_{L^2(\mathbb R^n\times (-\rho, \rho))} \nonumber \\
&\geq 2\rho \|\int^\mu_{0}  dP_{\lambda}f\|^2_{L^2(\mathbb R^n)}\nonumber \\
&+\| \int^0_{-\infty} \cos(\sqrt{-\lambda}s) dP_{\lambda}f\|^2_{L^2(\mathbb R^n\times (-\rho, \rho))}\nonumber \\
&\geq C\rho \|\mathbb{I}_{\mu}(f)\|^2_{L^2(\mathbb R^n)},
\end{align*}
where we have chosen $\rho$ small such that $\cos(\sqrt{-\lambda}\rho)>c_0$ as in (\ref{315}). Thus, we achieve 
\begin{align*}
C\rho\|\mathbb{I}_{\mu}(f)\|_{L^2(\mathbb R^n)} \leq \|  \int^\mu_{-\infty}\mathcal{S}_\lambda(s) dP_{\lambda}f\|_{H^1(\mathbb R^n\times (-\rho, \rho))}.
\end{align*}

Next we estimate the second inequality in (\ref{compare-in}). 
We split the integration in term of positive spectrum $\lambda$ and negative spectrum $\lambda$. It follows from (\ref{calcu-1}) that
\begin{align}
\| \int^\mu_{-\infty} \mathcal{S}_\lambda(s) dP_{\lambda}f\|^2_{L^2(\mathbb R^n)}
& \leq \| \int^\mu_{0} \frac{\sinh(\sqrt{\lambda} s)}{\sqrt{\lambda}} dP_{\lambda}f\|^2_{L^2(\mathbb R^n)} +\| \int^{0}_{-\infty} |s| dP_{\lambda}f\|^2_{L^2(\mathbb R^n)}  \nonumber \\
& \leq (1+e^{2|\mu||s|}s^2 )\|\mathbb{I}_{\mu}(f)\|^2_{L^2(\mathbb R^n)}.
\label{3.15}
\end{align}
Then integrating with respect to $s$ variable gives that
\begin{align}
\| \int^\mu_{-\infty} \mathcal{S}_\lambda(s) dP_{\lambda}f\|^2_{L^2(\mathbb R^n\times (-\rho, \rho))}& \leq C\rho^3 e^{2\rho|\mu|} \|\mathbb{I}_{\mu}(f)\|^2_{L^2(\mathbb R^n)}.
\label{integration-l2}
\end{align}
We further consider the $L^2$ norm with respect to $s$ derivative. By the property (\ref{calcu-2}), we get
\begin{align*}
\| \int^\mu_{-\infty} \partial_s \mathcal{S}_\lambda(s) dP_{\lambda}f\|^2_{L^2(\mathbb R^n)} 
& \leq \| \int^\mu_{0} \cosh(\sqrt{\lambda} s) dP_{\lambda}f\|^2_{L^2(\mathbb R^n)}  +\| \int^{0}_{-\infty}1  dP_{\lambda}f\|^2_{L^2(\mathbb R^n)}   \nonumber \\
& \leq e^{2|\mu||s|} \|\mathbb{I}_{\mu}(f)\|_{L^2(\mathbb R^n)}.
\end{align*}
Thus, 
\begin{align}
\| \int^\mu_{-\infty} \partial_s \mathcal{S}_\lambda(s) dP_{\lambda}f\|^2_{L^2(\mathbb R^n\times (-\rho, \rho))} \leq 2\rho e^{2|\mu|\rho} \|\mathbb{I}_{\mu}(f)\|_{L^2(\mathbb R^n)}.
\label{integration-s}
\end{align}

Next we consider the $L^2$ norm of weak derivative of $\int^\mu_{-\infty}\mathcal{S}_\lambda(s) d P_{\lambda}f$ with spatial variables. Let $\triangle_x$ the Laplace operator with respect to $x$ variables. It holds that
\begin{align*}
-\triangle_x \int^\mu_{-\infty}\mathcal{S}_\lambda(s) d P_{\lambda}f+V(x) \int^\mu_{-\infty}\mathcal{S}_\lambda(s) d P_{\lambda}f=\int^\mu_{-\infty}\lambda\mathcal{S}_\lambda(s) d P_{\lambda}f.
\end{align*}

Multiplying both sides by  $\int^\mu_{-\infty}\mathcal{S}_\lambda(s) d P_{\lambda}f$ and integrating by parts gives that
\begin{align*}
\| \nabla_x \int^\mu_{-\infty}\mathcal{S}_\lambda(s) d P_{\lambda}f\|^2_{L^2(\mathbb R^n)} 
&\leq \|\int^\mu_0 {\sinh(\sqrt{\lambda}s)}  dP_{\lambda}f\|^2_{L^2(\mathbb R^n)}+\|\int^0_{-\infty} {\sinh(i\sqrt{-\lambda}s)}  dP_{\lambda}f\|^2_{L^2(\mathbb R^n)}
\nonumber \\
&+\|V\|_{L^\infty} \|\int^\mu_{-\infty}\mathcal{S}_\lambda(s) d P_{\lambda}f\|^2_{L^2(\mathbb R^n)} \nonumber \\
&\leq (|\mu|+1) s^2 e^{2 |\mu||s|} \|\int^\mu_{-\infty}   dP_{\lambda}f\|^2_{L^2(\mathbb R^n)} \nonumber \\
& +s^2\|V\|_{L^\infty}  e^{2 |\mu||s|} \|\int^\mu_{-\infty}   dP_{\lambda}f\|^2_{L^2(\mathbb R^n)}\nonumber \\
&\leq C(1+|\mu|)s^2 e^{2|\mu| |s|}  \|\mathbb{I}_{\mu}(f)\|^2_{L^2(\mathbb R^n)},
\end{align*}
where we used the property (\ref{calcu-1}), (\ref{3.15}) and boundedness of $V(x)$.

Integration with respect to $\rho$ yields that
\begin{align}
\| \nabla_x \int^\mu_{-\infty}\mathcal{S}_\lambda(s) dP_{\lambda}f\|^2_{L^2(\mathbb R^n\times (-\rho, \rho))} 
\leq C(|\mu|+1) \rho^3 e^{2|\mu| \rho}  \|\mathbb{I}_{\mu}(f)\|^2_{L^2(\mathbb R^n)}.
\label{integration-x}
\end{align}

Taking (\ref{integration-l2}), (\ref{integration-s}), and (\ref{integration-x}) in account  gives that
\begin{align*}
\|  \int^\mu_{-\infty}\mathcal{S}_\lambda(s) dP_{\lambda}f\|^2_{H^1(\mathbb R^n\times (-\rho, \rho))}\leq C\rho(1+{\rho^2}(1+|\mu|) )
e^{2\rho |\mu|} \|\mathbb{I}_{\mu}(f)\|^2_{L^2(\mathbb R^n)}.
\end{align*}
Therefore, the lemma is arrived.
\end{proof}

We are ready to give the proof of Theorem \ref{th2}. We use the quantitative results on propagation of smallness of gradients in Corollary \ref{cor-1-1}.
\begin{proof}[Proof of Theorem \ref{th2}]
In the following, we adopt these notations.  We still have $Q_j=\Lambda_2(j)$ and $\Omega_j=\Omega \cap \Lambda_1(j)$ for any $j\in \mathbb Z^n$. Denote $\mathcal{E}_j=\Omega_j\times \{0\}\times (-\frac{\rho}{2}, \frac{\rho}{2} )\times (-\frac{\rho}{2}, \frac{\rho}{2} )\times (-\frac{\rho}{2}, \frac{\rho}{2} )$ and
$\mathcal{Q}_j=Q_j\times (-{\rho},{\rho} )\times  (-{\rho},{\rho} )\times (-{\rho},{\rho} )\times (-{\rho},{\rho} ) $. 
As in the proof of Theorem \ref{th1}, we can identify the cubes $\mathcal{Q}_j$ with the standard cube $Q^{n+4}$ with size $2$,  and the measurable set $\mathcal{E}_j$ with the set  $E_1$ in Corollary \ref{cor-1-1}.
Let 
\begin{align}
\hat{f}=e^{ \sqrt{C_0} t} e^{\sqrt{\hat{C}_0} y}e^{i\sqrt{\hat{C}_0} \tau}\int^\mu_{-\infty} d P_{\lambda}f \ \mbox{and} \ \hat{f}_1=e^{ \sqrt{C_0} t} e^{\sqrt{\hat{C}_0} y}e^{i\sqrt{\hat{C}_0} \tau}w^{-1}\int^\mu_{-\infty} d P_{\lambda}f.
\label{hat-f}
\end{align}
Recall that $\bar f$ in (\ref{bar-f}). Note that
$$\nabla \bar f (x,0,t,y,\tau)=\langle 0, \hat{f}_1, 0, 0, 0\rangle.$$ 

%From (\ref{second-ell-1-}), $\bar f$ satisfies %(\ref{elliptic}). It is easy to see that (\ref{metric}) and (\ref{metric-1}) are satisfied.
%\begin{align*}
%\frac{1}{2C_0} \leq \frac{C_0-V(x)}{{\|V\|_{C^{0, 1}}}+C_0}\leq 1, \quad  \|\frac{C_0-V(x)}{{\|V\|_{C^{0, 1}}}+C_0}\|_{C^{0, 1} }\leq 1
%\end{align*}

From (\ref{elli-1-1}) and (\ref{elli-2-1}), we see that the leading coefficients $A(x, s,t, y,\tau)$ in (\ref{second-elli-1-1}) satisfies (\ref{metric-1}) and (\ref{metric}) in $6\mathcal{Q}j$. Thus, (\ref{logunov-mal-2-3}) is applicable with $E_1=\mathcal{E}_j$ and $Q^{n+4}=\mathcal{Q}_j$. Replacing $\Phi_1$ by  $\bar f$  and $\hat{\phi}_1$ by $\hat{f}_1$  in (\ref{logunov-mal-2-3}) gives that
\begin{align}
\|\nabla \bar f \|_{L^2 ( \mathcal{Q}_j)}\leq C (\frac{|\mathcal{Q}_j|}{|\mathcal{E}_j|})^{\frac{\gamma}{2}} \|\hat{f}_1\|^\gamma _{L^2(\mathcal{E}_j)} \|\nabla \bar f \|^{1-\gamma}_{L^2( 4\mathcal{Q}_j)},
\label{three-second}
\end{align}
where $\gamma =\frac{1}{\frac{1}{C_2}\ln \frac{C_1|\mathcal{Q}_j|^{\frac{n+3}{n+4}}}{|\mathcal{E}_j|}+1}.$
It follows from  (\ref{elli-1-1}) that
\begin{align*}
   \|\hat{f}_1\|_{L^2(\mathcal{E}_j)}&\leq C \|\hat{f}\|_{L^2(\mathcal{E}_j)} 
   \leq C e^{(\sqrt{C_0}+\sqrt{\hat{C}_0})\rho} \|\int^\mu_{-\infty} d P_{\lambda}f\|_{L^2(\Omega_j)}.
\end{align*}

Recall $\check{f}$ in (\ref{check-f}).   It follows from (\ref{brain-1-1}) and  (\ref{brain-2-1}) that
\begin{align}
\| \check{f} \|_{H^1 ( {Q}_j\times (-\rho, \rho)}\leq C e^{(\sqrt{C_0}+\sqrt{\hat{C}_0})\gamma\rho} (\frac{|\mathcal{Q}_j|}{|\mathcal{E}_j|})^{\frac{\gamma}{2}} \|\mathbb{I}_{\mu}(f)\|^\gamma _{L^2(\Omega_j)} \|\check{f} \|^{1-\gamma}_{H^1( 4{Q}_j\times (-4\rho, 4\rho))}.
\label{339}
\end{align}
From the value of $\gamma$ and the definition of sensors sets $\Omega$ defined in (\ref{geom-om-2}), we have
$$\frac{|\mathcal{Q}_j|}{|\mathcal{E}_j|}\leq \frac{1}{\delta} $$
and
$$\frac {1}{\gamma}\leq \frac{1}{C_2} \ln C_1 (\frac{1}{\delta}).$$ 
We can choose 
\begin{align}
    {\gamma}= {C_2} \big(\ln C_1 (\frac{1}{\delta})\big)^{-1}
\end{align}
as (\ref{339}) still holds for such a $\gamma$.
Let ${\Omega}=\cup \Omega_j$.
As $\mathbb R^n\subset \cup_{j\in \mathbb{Z}^n} Q_j$, summing up the estimate (\ref{339}) in ${Q}_j\times (-{\rho},{\rho}) $ and using the H\"older's inequality (\ref{holder}) by choosing
\begin{align*}
    a_j=\|\hat{f}\|_{L^2(\Omega_j)}, \quad b_j=\| \check{f} \|_{H^1( 4{Q}_j\times (-4\rho, 4\rho))},
\end{align*}
 we get
\begin{align}
\| \check{f}  \|_{H^1 (\mathbb R^n\times(-{\rho},{\rho} ) ) }\leq Ce^{(\sqrt{C_0}+\sqrt{\hat{C}_0})\gamma\rho} (\frac {1}{\delta})^{\frac{\gamma}{2}}  \|\mathbb{I}_{\mu}(f)\|^\gamma _{L^2({\Omega})} \|\check{f}  \|^{1-\gamma}_{H^1( \mathbb R^n\times(-4{\rho},4{\rho} ))}.
\label{global-three}
\end{align}
The application of  (\ref{compare-in}) in  Lemma \ref{lem-com} yields that
\begin{align*}
C\rho\|\mathbb{I}_{\mu}(f)\|_{L^2(\mathbb R^n)} \leq  C(1+|\mu|)e^{(\sqrt{C_0}+\sqrt{\hat{C}_0})\gamma\rho}  e^{(1-\gamma)(\sqrt{|\mu|}+1)\rho}
  (\frac {1}{\delta})^{\frac{\gamma}{2}}  \|\mathbb{I}_{\mu}(f)\|^\gamma_{L^2(\Omega)} \|\mathbb{I}_{\mu}(f)\|^{1-\gamma}_{L^2(\mathbb R^n)}.
\end{align*}
By fixing $\rho$ as satisfied in Lemma \ref{lem-com} and using bounded assumption of $V(x)$ in { Assumption (B)}, we obtain that
\begin{align*}
\|\mathbb{I}_{\mu}(f)\|_{L^2(\mathbb R^n)}\leq e^{C((\frac{1}{\gamma}-1)(\sqrt{|\mu|}+1)+{1})}(\frac {1}{\delta})^{\frac{1}{2}} 
 \|\mathbb{I}_{\mu}(f)\|_{L^2(\Omega)}.
\end{align*}
Hence, we arrive at
\begin{align*}
\|\mathbb{I}_{\mu}(f)\|_{L^2(\mathbb R^n)}\leq (\frac{1}{\delta})^{C(\sqrt{|\mu|}+1)}\|\mathbb{I}_{\mu}(f)\|_{L^2(\Omega)}.
\end{align*}
Therefore, the proof of the theorem is completed.
\end{proof}

%Something more to think because we used the following for $\sigma=0$. \begin{align} 
%\frac{|\Omega\cap \Lambda_d(k) |}{|\Lambda_d(k)|}\geq \delta^{1+|k|^\sigma}
%\end{align}

\section{Observability inequality}
In this section, we will apply the spectral inequalities to  show  observability inequalities (\ref{aim-res}) on sensor sets for  heat equations in Theorem \ref{th3}. Especially, the observation region is concerned with a subset of positive measure in $(0, T)$. To this end, we need
the following lemma which is on the property of density points for sets of positive measure on $(0, T)$, see e.g. \cite{PW13}.
\begin{lemma}
Let $J$ be a subset of positive measure in $(0, T)$ and $k$ be a density point of $J$. Then for any $\alpha>1$, there exists $k_1\in (k, T)$ such that the sequence defined by
\begin{align*}
k_{m+1}-k= \alpha^{-m}(k_1-k)
\end{align*}
satisfies
\begin{align*}
|J\cap (k_{m+1}, k_m)|\geq \frac{ (k_m-k_{m+1})}{3}.
\end{align*}
\label{densi}
\end{lemma}

Next we deduce the sharp observability inequality in Theorem \ref{th3} from spectral inequality (\ref{aim-res}). We follow the proof in \cite{AEWZ14}, which in turn relies on the ideas in \cite{M10}, interpolation inequalities and telescopic series method in \cite{PW13}.

\begin{proof}[Proof of Theorem \ref{th3}:]
Let $\tilde{H}=\triangle-V(x)$ and $\tilde{P}_\lambda$ be the projection onto the space generated by $\{\phi_k: \lambda_k\leq \lambda\}$, where $\phi_k$ is the eigenfunction corresponding to $\lambda_k$ associated with $\tilde{H}$. 
Then $\tilde{P}^\lambda=I-\tilde{P}_\lambda$. Note that $$e^{t\tilde{H}} f=\sum a_i e^{-\lambda_i t} \phi_i(x),$$ where $a_i=\int_{\mathbb R^n} f\phi_i$. Thanks to the spectral inequality (\ref{aim-res}) in Theorem \ref{th1},
we have
\begin{align}
\| e^{t\tilde{H}} f\|_{L^2(\mathbb R^n)}&\leq \| e^{t\tilde{H}} \tilde{P}_\lambda f\|_{L^2(\mathbb R^n)}+ \| e^{t\tilde{H}} \tilde{P}^\lambda f\|_{L^2(\mathbb R^n)} \nonumber \\
&\leq (\frac{1}{\delta})^{C\lambda^{\frac{\sigma}{\beta_1}+\frac{1}{2}}}   \| e^{t\tilde{H}} \tilde{P}_\lambda  f \|_{L^2(\Omega)} + \| e^{t\tilde{H}} \tilde{P}^\lambda f\|_{L^2(\mathbb R^n)}\nonumber \\
&\leq (\frac{1}{\delta})^{C\lambda^{\frac{\sigma}{\beta_1}+\frac{1}{2}}} ( \| e^{t\tilde{H}}  f\|_{L^2(\Omega)}+ \|e^{t\tilde{H}} \tilde{P}^\lambda  f\|_{L^2(\mathbb R^n)} ) + \| e^{t\tilde{H}} \tilde{P}^\lambda f \|_{L^2(\mathbb R^n)}  \nonumber \\
&\leq (\frac{1}{\delta})^{C\lambda^{\frac{\sigma}{\beta_1}+\frac{1}{2}}}   ( \| e^{t\tilde{H}}  f\|_{L^2(\Omega)}+ \|e^{t\tilde{H}} \tilde{P}^\lambda  f\|_{L^2(\mathbb R^n)} ).\nonumber
\end{align}
With $t>s$, we have
\begin{align}
\| e^{t\tilde{H}} f\|_{L^2(\mathbb R^n)} &\leq(\frac{1}{\delta})^{C\lambda^{\frac{\sigma}{\beta_1}+\frac{1}{2}}}  ( \| e^{t\tilde{H}}  f\|_{L^2(\Omega)}+ e^{-\lambda(t-s)} \|e^{s\tilde{H}} \tilde{P}^\lambda  f\|_{L^2(\mathbb R^n)} )
\label{star-star}
\end{align}
for $\lambda\geq 0$.
We choose some constant $0<\tau<1$, which to be determined.  We can show that
\begin{align*}
\sup_{\lambda \geq 0}  e^{C\lambda^{\frac{\sigma}{\beta_1}+\frac{1}{2}}\ln \frac{1}{\delta}- \tau \lambda(t-s)} = e^{C_2(\tau(t-s))^{\frac{\frac{\sigma}{\beta_1}+\frac{1}{2}}{ {\frac{\sigma}{\beta_1}-\frac{1}{2}}}} (\ln \frac{1}{\delta})^{ \frac{-1} { \frac{\sigma}{\beta_1}-\frac{1}{2}}} },
\end{align*}
where the supremum is achieved by $\lambda=(\frac{\tau(t-s)  }{C(\frac{\sigma}{\beta_1}+\frac{1}{2}) \ln \frac{1}{\delta}})^{ \frac{1} { \frac{\sigma}{\beta_1}-\frac{1}{2}} }$ and $C_2$ depends on $\beta_1$, $\sigma$. $c_1$, $c_2$ and $n$.
Denote
\begin{align*}
\sigma_1= \frac{\sigma}{\beta_1}+\frac{1}{2}, \quad \quad \sigma_2= \frac{1}{2}-\frac{\sigma}{\beta_1}.
\end{align*}
Note that $\sigma_1>0$ and $\sigma_2>0$.
It follows from (\ref{star-star}) that
\begin{align*}
\| e^{t\tilde{H}} f\|_{L^2(\mathbb R^n)}&\leq e^{C_2(\tau(t-s))^{\frac{-\sigma_1}{\sigma_2} }( \ln \frac{1}{\delta})^{ \frac{1} {\sigma_2} }} (  e^{\tau \lambda (t-s)} \| e^{t\tilde{H}}  f\|_{L^2(\Omega)}+ e^{-(1-\tau)\lambda(t-s)} \|e^{s\tilde{H}}   f\|_{L^2(\mathbb R^n)}).
\end{align*}
We aim to minimize the right hand side of the last inequality. Since $\lambda$ is some positive free parameter, we choose $\lambda$ such that
\begin{align*}
 e^{\lambda(t-s)}=\frac{ \| e^{s\tilde{H}} f\|_{L^2(\mathbb R^n)}} {  \| e^{t\tilde{H}} f\|_{L^2(\Omega)}}.
\end{align*}
Thus, we obtain
\begin{align*}
\| e^{t\tilde{H}} f\|_{L^2(\mathbb R^n)}\leq 2 e^{C_2(\tau(t-s))^{\frac{-\sigma_1}{\sigma_2} }(\ln \frac{1}{\delta})^{ \frac{1} {\sigma_2} }} \| e^{t\tilde{H}} f\|^{1-\tau }_{L^2(\Omega)}  \| e^{s\tilde{H}} f\|^{\tau}_{L^2(\mathbb R^n)}.
\end{align*}
In particular, if $s=0$, we have
\begin{align}
\| u(x,t)\|_{L^2(\mathbb R^n)}\leq 2e^{C_2(\tau t)^{\frac{-\sigma_1}{\sigma_2} }( \ln \frac{1}{\delta})^{ \frac{1} {\sigma_2} }} \| u(x,t)\|^{1-\tau }_{L^2(\Omega)}  \| u(x,0)\|^{\tau}_{L^2(\mathbb R^n)}.
\label{heat-u}
\end{align}
This is a quantitative unique continuation result for the heat equations (\ref{nonlinear-one}). It is well-known that $u(x,t)=0$ in $\mathbb R^n\times (0, T)$
if $u(x,T)=0$ or  $u(x,0)=0$ in $\mathbb R^n$. From (\ref{heat-u}), we learn that $u(x,t)=0$ in $\mathbb R^n\times (0, T)$
if $u(x,T)=0$ in $\Omega$, where $\Omega$ is given in (\ref{geom-om}).

We will consider $t\in [t_1, t_2]$. Choosing $s=t_1$, by energy estimates, we get
\begin{align*}
\| e^{t_2\tilde{H}} f\|_{L^2(\mathbb R^n)}&\leq e^{T\|V^-\|_{\infty}} \| e^{t\tilde{H}} f\|_{L^2(\mathbb R^n)} \nonumber \\
&\leq  C e^{T\|V^-\|_{\infty}} e^{C_2(\tau(t-t_1))^{\frac{-\sigma_1}{\sigma_2} }(\ln \frac{1}{\delta})^{ \frac{1} {\sigma_2} }}\| e^{t\tilde{H}} f\|^{1-\tau}_{L^2(\Omega)}  \| e^{t_1\tilde{H}} f\|^{\tau}_{L^2(\mathbb R^n)}.
\end{align*}
Let $\alpha=2$, $t_1=k_{m+1}$ and $t_2=k_m$. From Lemma \ref{densi}, we get that  $|J\cap  (t_1,  t_2)|\geq \frac{t_2-t_1}{3}$. Then it holds that
\begin{align*}
 |J\cap  (t_1+\frac{t_2-t_1}{4}, t_2)|\geq \frac{t_2-t_1}{12}.
 \end{align*}
Integrating the above inequality for $t\in J\cap (t_1+\frac{t_2-t_1}{4}, t_2)$ and using H\"older's inequality yields that
\begin{align*}
\| e^{t_2\tilde{H}} f\|_{L^2(\mathbb R^n)}
&\leq  C e^{T\|V^-\|_{\infty}} e^{C_2(\frac{1}{4}\tau(t_2-t_1))^{\frac{-\sigma_1}{\sigma_2} }(\ln \frac{1}{\delta})^{ \frac{1} {\sigma_2} }}
\nonumber \\
&\times ( \int^{t_2}_{t_1+\frac{t_2-t_1}{4}}\mathrm{1}_J(t) \| e^{t\tilde{H}} f\|_{L^2(\Omega)}\, dt)^{{1-\tau}}  \| e^{t_1\tilde{H}} f\|^{\tau}_{L^2(\mathbb R^n)}.
\end{align*}
Let $a$ be some positive constant, which is to be determined. By Young's inequality, $AB\leq (1-\tau)A^{\frac{1}{1-\tau}}+\tau B^{\frac{1}{\tau}}$, we get
\begin{align*}
e^{-a (t_2-t_1)^{\frac{-\sigma_1}{\sigma_2} }( \ln \frac{1}{\delta})^{ \frac{1} {\sigma_2}} }\| e^{t_2\tilde{H}} f\|_{L^2(\mathbb R^n)}&\leq C e^{T\|V^-\|_{\infty}} e^{- \frac a 2 (t_2-t_1)^{\frac{-\sigma_1}{\sigma_2} } ( \ln \frac{1}{\delta})^{ \frac{1} {\sigma_2}} }\nonumber \\
&\times 
( \int^{t_2}_{t_1+\frac{t_2-t_1}{4}} \mathrm{1}_E(t)\| e^{t\tilde{H}} f\|_{L^2(\Omega)}\, dt)^{{1-\tau}}   \nonumber \\ &\times  e^{ \big(C_2 (\frac{\tau}{4})^{\frac{-\sigma_1}{\sigma_2} }  -\frac a 2 \big) (t_2-t_1)^{\frac{-\sigma_1}{\sigma_2} }(\ln \frac{1}{\delta})^{ \frac{1} {\sigma_2}} }   \| e^{t_1\tilde{H}} f\|^{\tau}_{L^2(\mathbb R^n)}\nonumber \\
&\leq  C (1-\tau) e^{\frac{T\|V^-\|_{\infty}}{1-\tau}} e^{\frac{-a/2}{1-\tau} (t_2-t_1)^{\frac{-\sigma_1}{\sigma_2} }( \ln \frac{1}{\delta})^{ \frac{1} {\sigma_2}} } \nonumber \\
&\times ( \int^{t_2}_{t_1+\frac{t_2-t_1}{4}}\mathrm{1}_J(t) \| e^{t\tilde{H}} f\|_{L^2(\Omega)}\, d t) \nonumber \\ & + \tau e^{\frac{1}{\tau} \big(C_2 (\frac{\tau}{4})^{\frac{-\sigma_1}{\sigma_2}} - \frac a 2 \big) (t_2-t_1)^{\frac{-\sigma_1}{\sigma_2} }( \ln \frac{1}{\delta})^{ \frac{1} {\sigma_2}}  }  \| e^{t_1\tilde{H}} f\|_{L^2(\mathbb R^n)}.
\end{align*} 
Now we choose the constants $a$ and $\tau$ such that
\begin{align*}
\frac{a}{2}\geq d_0 \quad \mbox{and} \quad  \frac{1}{\tau} \big(C_2 (\frac{\tau}{4})^{\frac{-\sigma_1}{\sigma_2}} -\frac a 2 \big)     >2^{\frac{\sigma_1}{\sigma_2}} a
\end{align*}
for some positive constant $d_0$.
Hence, we have
\begin{align*}
e^{-a (t_2-t_1)^{\frac{-\sigma_1}{\sigma_2} }(\ln \frac{1}{\delta})^{ \frac{1} {\sigma_2}} }\| e^{t_2\tilde{H}} f\|_{L^2(\mathbb R^n)}&\leq C e^{C_1 T\|V^-\|_{\infty}} e^{-d_0 (t_2-t_1)^{\frac{-\sigma_1}{\sigma_2} }(\ln \frac{1}{\delta})^{ \frac{1} {\sigma_2}} } \nonumber \\
&\times\int^{t_2}_{t_1+\frac{t_2-t_1}{4}}  \mathrm{1}_J(t) \| e^{t\tilde{H}} f\|_{L^2(\Omega)}\, dt \nonumber \\& +e^{-a (\frac{t_2-t_1}{2})^{\frac{-\sigma_1}{\sigma_2} }( \ln \frac{1}{\delta})^{ \frac{1} {\sigma_2}} }\| e^{t_1\tilde{H}} f\|_{L^2(\mathbb R^n)},
\end{align*}
where $C_1$ depends on the constants in the Assumption (A).
Furthermore,
\begin{align*}
e^{-a (t_2-t_1)^{\frac{-\sigma_1}{\sigma_2} }( \ln \frac{1}{\delta})^{ \frac{1} {\sigma_2}} }\| e^{t_2\tilde{H}} f\|_{L^2(\mathbb R^n)}&-e^{-a (\frac{t_2-t_1}{2})^{\frac{-\sigma_1}{\sigma_2} }( \ln \frac{1}{\delta})^{ \frac{1} {\sigma_2}} }\| e^{t_1\tilde{H}} f\|_{L^2(\mathbb R^n)} \nonumber \\&\leq C e^{C_1T\|V^-\|_{\infty}} 
\int^{t_2}_{t_1+\frac{t_2-t_1}{4}}  \mathrm{1}_J(t) \| e^{t\tilde{H}} f\|_{L^2(\Omega)}\, dt.
\end{align*}
 Recall that we have chosen $\alpha=2$, $t_1=k_{m+1}$ and $t_2=k_{m}$. Observe that
\begin{align*}
\frac{1}{k_{m+1}-k_{m+2}}=\frac{2}{k_{m}-k_{m+1}}.
\end{align*}

We have
\begin{align*}
e^{-a (k_m-k_{m+1})^{\frac{-\sigma_1}{\sigma_2} }(\ln \frac{1}{\delta})^{ \frac{1} {\sigma_2}} }\| e^{k_m\tilde{H}} f\|_{L^2(\mathbb R^n)} &-e^{-a (k_{m+1}-k_{m+2})^{\frac{-\sigma_1}{\sigma_2} }( \ln \frac{1}{\delta})^{ \frac{1} {\sigma_2}} }\| e^{k_{m+1}\tilde{H}} f\|_{L^2(\mathbb R^n)} \nonumber \\ & \leq C e^{C_1T\|V^-\|_{\infty}} 
\int^{k_m}_{k_{m+1}} \mathrm{1}_J(t)\| e^{t\tilde{H}} f\|_{L^2(\Omega)}\, dt.
\end{align*}
 As $m\to \infty$,  it holds that $$e^{-a (k_m-k_{m+1})^{\frac{-\sigma_1}{\sigma_2} }( \ln \frac{1}{\delta})^{ \frac{1} {\sigma_2}} }\to 0.$$
 We  sum up the above telescopic series from $m=1$ to $\infty$ to have
\begin{align*}
e^{-a (k_1-k_{2})^{\frac{-\sigma_1}{\sigma_2} }( \ln \frac{1}{\delta})^{ \frac{1} {\sigma_2}} }\| e^{k_1\tilde{H}} f\|_{L^2(\mathbb R^n)}\leq   C e^{C_1T\|V^-\|_{\infty}}  \int^{k_1}_{k} \mathrm{1}_J(t)\| e^{t\tilde{H}} f\|_{L^2(\Omega)}\, dt.
\end{align*}
Then 
\begin{align*}
\| e^{T\tilde{H}} f\|_{L^2(\mathbb R^n)}\leq   C(J) e^{C_1T\|V^-\|_{\infty}}   e^{ C(J) ( \ln \frac{1}{\delta})^{ \frac{1} {\sigma_2}} }
\int_{J} \| e^{t\tilde{H}} f\|_{L^2(\Omega)}\, dt.
\end{align*}
Let $f=u_0$ in (\ref{nonlinear-one}). We arrive at
\begin{align*}
\| u(x, T)\|_{L^2(\mathbb R^n)}\leq  C(J) e^{C_1T\|V^-\|_{\infty}}   e^{ C(J) (\ln \frac{1}{\delta})^{ \frac{1} {\sigma_2}} }
(\int_J \int_\Omega u^2 \, dx dt)^\frac{1}{2}.
\end{align*}
Therefore, we finish the proof of Theorem \ref{th3}.
\end{proof}

\end{document}